\documentclass[11pt,a4paper]{article}
\newcommand{\inc}{\mrm{in}}
\usepackage{titlesec}
\usepackage{fancyhdr}
\usepackage{a4wide}
\usepackage{graphicx}
\usepackage{float}
\usepackage{amssymb}
\usepackage{amsmath}
\usepackage{amsthm}
\usepackage{color}
\usepackage{mathrsfs}
\usepackage[T1]{fontenc}
\usepackage{inputenc}
\usepackage{tabularx}
\usepackage[english]{babel}
\usepackage{tikz}
\usepackage{pgf}
\usetikzlibrary{arrows}
\usepackage{lmodern}
\usepackage{hyperref}
\usepackage{geometry}
\usepackage{changepage}
\changepage{0pt}{}{}{}{}{0pt}{}{0pt}{5pt}
\usepackage[numbers]{natbib}
\DeclareGraphicsExtensions{.jpg,.png,.JPG,.pdf,.svg} 
\DeclareGraphicsExtensions{.eps,.pdf,.png,.jpg,.JPG,.jpeg,.ps}
\DeclareTextSymbol{\degre}{T1}{6}
\usepackage{ textcomp }

\usepackage{hyperref}
\hypersetup{
pdfpagemode=none,
pdftoolbar=true,        
pdfmenubar=true,        
pdffitwindow=false,     
pdfstartview={Fit},    
pdftitle={On the use of Perfectly Matched Layers at corners for scattering problems with sign-changing coefficients},    
pdfauthor={A.-S. Bonnet-Ben Dhia, C. Carvalho, L. Chesnel, P. Ciarlet Jr.},     
pdfcreator={A.-S. Bonnet-Ben Dhia, C. Carvalho, L. Chesnel, P. Ciarlet Jr.},   
pdfproducer={A.-S. Bonnet-Ben Dhia, C. Carvalho, L. Chesnel, P. Ciarlet Jr.}, 
pdfkeywords={}, 
pdfnewwindow=true,      
colorlinks=true,       
linkcolor=magenta,          
citecolor=red,        
filecolor=cyan,      
urlcolor=blue           
}

\newcommand{\Ls}{\Lambda_{\text{sym}}}
\newcommand{\Lss}{\Lambda_{\text{skew}}}
\newcommand{\ke}{\kappa_{\eps}} 
\newcommand{\bfc}{\mathbf{c}}
\newcommand{\e}[1]
{\eps_#1}
\newcommand{\diff}[2]
{
\displaystyle{\frac{\partial #1}{\partial #2}}
}                                                              
\newcommand{\bfk}{\mathbf{k}}
\newcommand{\me}{\text{m}} 
\newcommand{\di}{\text{d}} 
\newcommand{\eC}{\mathbb{C}}
\newcommand{\eR}{\mathbb{R}}
\newcommand{\curl}{{\text{curl }}} 
\newcommand{\tT}{\mathtt{T}}
\newcommand{\Frac}{\displaystyle\frac}
\newcommand{\inv}{^{-1}}
\newcommand{\grad}{\nabla}
\newcommand{\dsp}{\displaystyle}

\newcommand{\eps}{\varepsilon}
\newcommand{\om}{\omega}
\newcommand{\Om}{\Omega}
\newcommand{\mrm}[1]{\mathrm{#1}}

\newcommand{\bfx}{\boldsymbol{x}}

\newcommand{\Cplx}{\mathbb{C}}

\newcommand{\R}{\mathbb{R}}

\renewcommand{\div}{\mrm{div}}

\newcommand{\out}{\mrm{out}}
\newcommand{\ext}{\mrm{ext}}

\newtheorem{lemma}{Lemma}[section]
\newtheorem{remark}{Remark}[section]

\newtheorem{proposition}{Proposition}[section]

\newtheorem{lemmaAnnexe}{Lemma}

\begin{document}

~\vspace{0.0cm}
\begin{center}
{\sc \bf\LARGE On the use of Perfectly Matched Layers at corners\\[7pt] for scattering problems with sign-changing coefficients}
\end{center}

\begin{center}
\textsc{Anne-Sophie Bonnet-Ben Dhia}$^1$, \textsc{Camille Carvalho}$^{1}$, \textsc{Lucas Chesnel}$^2$, \textsc{Patrick Ciarlet}$^1$\\[16pt]
\begin{minipage}{0.9\textwidth}
{\small
$^1$ Laboratoire Poems, UMR 7231 CNRS/ENSTA/INRIA, Ensta ParisTech, 828, Boulevard des Maréchaux, 91762 Palaiseau, France; \\
$^2$ INRIA/Centre de mathématiques appliquées, \'Ecole Polytechnique, Université Paris-Saclay, Route de Saclay, 91128 Palaiseau, France.\\[10pt] 
E-mails: \texttt{Bonnet-Bendhia@ensta-paristech.fr}, \texttt{Carvalho@ensta-paristech.fr}, \texttt{Lucas.Chesnel@inria.fr}, \texttt{Ciarlet@ensta-paristech.fr}\\[-14pt]
\begin{center}
(\today)
\end{center}
}
\end{minipage}
\end{center}
\vspace{0.4cm}

\noindent\textbf{Abstract.} 
We investigate in a $2$D setting the scattering of time-harmonic electromagnetic waves by a plasmonic device, represented as a non dissipative bounded and penetrable obstacle with a negative permittivity. Using the $\tT $-coercivity approach, we first prove that the problem is well-posed in the classical framework $H^1_{\text{loc}} $ if the negative permittivity does not lie in some critical interval whose definition depends on the shape of the device. When the latter has corners, for values inside the critical interval, unusual strong singularities for the electromagnetic field can appear. In that case, well-posedness is obtained by imposing a radiation condition at the corners to select the outgoing black-hole plasmonic wave, that is the one which carries energy towards the corners. A simple and systematic criterion is given to define what is the outgoing solution. Finally, we propose an original numerical method based on the use of Perfectly Matched Layers at the corners. We emphasize that it is necessary to design an \textit{ad hoc} technique because the field is too singular to be captured with standard finite element methods.\\


\section{Introduction}

We are interested in the scattering of Transverse Magnetic (TM) time-harmonic electromagnetic waves by a metallic obstacle embedded in some dielectric medium, governed by the following scalar equation
\begin{equation}\label{eq:intro}
\div(\eps\inv \grad u ) + \omega^2 \mu u = 0,
\end{equation}
where $\omega $ is the frequency, $\eps $ is the dielectric permittivity and $\mu $ is the magnetic permeability. Unlike common materials, metals can exhibit a permittivity with a negative real part. More precisely, following the Drude's law (see e.g. \cite{Anan05}) the permittivity depends on the frequency:
\begin{equation}\label{lossyDrudeModel}
\eps (\om) = \eps_0\eps_r (\om) = \eps_0 \left( 1 - \Frac{\om^2_p}{\om^2 + i \om \gamma}\right).
\end{equation}
Here, $\eps_{0}>0 $ is the vacuum permittivity, $\eps_r (\om)$ is called the relative permittivity, $\gamma >0 $ characterizes the dissipative effects (we choose the convention of a harmonic regime in $e^{-i\om t}$), and $\om_p >0$ is the plasma frequency. Dissipation  becomes neglectable at frequencies $\om  $ such that $ \om \gg \gamma $, leading to the so-called dissipationless Drude's model: 
\begin{equation}\label{eq:drude}
\eps_r (\om) =  1 - \Frac{\om^2_p}{\om^2} .
\end{equation}
Then for $\omega < \omega_p $, $ \eps(\omega) = \eps_0 \eps_r(\om) $ takes negative real values. Note that, since $\gamma \ll \om_p $ for metals, the regime $\gamma \ll \om < \om_p$ is meaningful. Due to the change of sign of the dielectric constant at the interface metal-dielectric, resonances called surface plasmons can appear  \cite{Maier07,Raet80}. Over the past decades, surface plasmons revealed a great interest in guiding or confining light in nano-photonic devices \cite{BaDeEb03,GrBo10,ZaSM05}.

From a mathematical point of view, the study of Equation \eqref{eq:intro} with a function $\eps$ changing sign on the physical domain, has given rise to many contributions. In particular, an abstract mathematical approach named $\tT $-coercivity and based on variational methods has been proposed in \cite{BoCZ10,BoChCi12}. With this technique, it has been proved that Problem \eqref{eq:intro} set in a bounded domain supplemented with Dirichlet (or Neumann) boundary condition is of Fredholm type in the classical functional framework whenever the contrast (ratio of the values of $\eps$ across the interface) lies outside some interval $I_c $, called critical interval, which always contains the value $-1$. Moreover, this interval reduces to $\{-1\}$ if and if only the interface between the two materials is smooth (of class $\mathscr{C}^2$). Analogous results have been obtained by techniques of boundary integral equations in the 1980s in \cite{CoSt85} (see also \cite{Ola95,LeNg13}). Note that the critical value $-1 $ is associated through Equation \eqref{eq:drude} to the so-called surface plasmon frequency while the critical interval is associated to a critical range of frequencies. The numerical approximation of the solution of this scalar problem for a contrast outside the critical interval, based on classical finite element methods, has been investigated in \cite{BoCZ10,NiVe11,ChCi13}. Under some assumptions on the meshes, the discretized problem is well-posed and its solution converges to the solution of the continuous problem. Let us mention that the study of Maxwell's equations has been carried out in \cite{BoCC14,BoChCi12_Maxwell,BrGS15}.\\

For geometries with wedges and sharp corners, the solution exhibits strong singularities at these regions when the constrast is getting closer to the critical interval. This leads to a local energy enhancement of the light \cite{Stoc04,BVNMRB08,CMLZ09}. Even more, for a contrast inside the critical interval, the problem becomes ill-posed in the classical framework because the solution no longer belongs to $H^1 $. Up to now, mathematical analysis has addressed in $2$D the simplified electrostatic like equation
\begin{equation}\label{EquationChangeOfSign} 
\div\left( \eps^{-1}(\om) \nabla u \right) = 0,
\end{equation}
which can be seen as an approximation of \eqref{eq:intro} by zooming at a corner. It is interesting to note that the critical interval also appears in the studies of \eqref{EquationChangeOfSign} with techniques of conformal mappings \cite{ALFSMP10,Pendry10,FWGVMP12}. Concerning the mathematical framework, the influence of corners at the interface between the two materials has been clarified in \cite{BoChCl13} for equation \eqref{EquationChangeOfSign} set in a particular geometry (with one corner of particular aperture). In that case, when the contrast lies inside the critical interval, Fredholm property is lost because of the existence of two strongly oscillating singularities at the corner, responsible for the ill-posedness in the classical framework. These singularities can be interpreted as waves propagating towards or outwards the corner. Then selecting the outgoing singularity by means of a limiting absorption principle allows to recover Fredholmness of the problem.\\

The first goal of the present paper is to extend the theory to a more realistic scattering problem in free space, for a contrast of permittivities outside or inside $I_c$. The second objective is to present an original numerical method to approximate the solution when the contrast lies inside the critical interval. The approach, based on a finite element method, consists in using well-suited Perfectly Matched Layers (PMLs) at the corners to capture the strongly oscillating singularities. \\

It is important to emphasize that considering a problem without absorption, that is using model (\ref{eq:drude}) instead of model (\ref{lossyDrudeModel}) for $\eps_r(\om)$, is not only of mathematical interest. Indeed, we can show a limiting absorption principle (see in particular the discussion of \S\ref{subsectionOutgoing}): a solution of the scattering problem with small absorption behaves like a solution of the problem without absorption. Therefore, it is relevant to study the limit problem with a real valued $\eps_r(\om)$ like in (\ref{eq:drude}). For a numerical illustration of its relevance, see the beginning of Section \ref{sectionDiscussion}. More precisely, in Figure \ref{figs_dissip}, we observe that the technique we propose based on the use of PMLs at the corner, which is essential to approximate the solution without dissipation, is also necessary when losses are small.\\

This text is organized as follows. In section \ref{sec:pb_setting}, we define the problem and introduce an equivalent formulation set in a bounded domain using a classical Dirichlet-to-Neumann operator. Then for a contrast outside the critical interval, we prove it has a unique solution using the $\tT$-coercivity approach. In the rest of the paper, we consider the case of a contrast inside the critical interval. In section \ref{sec:new_framework}, we provide a detailed description of the singularities at corners. In particular, we give a systematic criterion to select the outgoing singularity which has to be taken into account through an adequate radiation condition at the corner. This condition yields a well-posed problem for a contrast inside the critical interval. Standard finite element methods fail to approximate the solution in the new framework because it is too singular. In section \ref{sec:PMLs}, we introduce an original numerical method to solve this problem. The idea is to use a coordinates transformation which maps a small disk around corners to semi-infinite strips. Then we implement Perfectly Matched Layers in the semi-infinite strips. We illustrate the method showing numerical results in the case of a triangular silver inclusion embedded in vacuum. Finally in section \ref{sectionDiscussion}, we conclude the paper with further discussions.

\section{Setting of the problem}\label{sec:pb_setting}

\subsection{The scattering problem with sign-changing permittivity}\label{ssec_pres_pb}

Propagation of time-harmonic electromagnetic waves in an inhomogeneous, isotropic, and lossless
medium is described by Maxwell's equations $i \om \eps \boldsymbol{E} + \textbf{\curl }\boldsymbol{H} = 0$ and $-i \om \mu\boldsymbol{H} + \textbf{\curl }\boldsymbol{E} = 0 $ for all $(\bfx,z)=(x,y,z)\in\R^3$. Here, $\boldsymbol{E}$, $\boldsymbol{H}$ correspond respectively to the electric and magnetic fields while $\eps$, $\mu$ are the dielectric permittivity and magnetic permeability. Introduce $\Om_{\text{m}}$ a bounded open set of $\R^2 $ with Lipschitz boundary $\Sigma :=\partial \Om_{\text{m}}$ and define $\Om_{\text{d}} := \eR^2 \setminus \overline{\Om_{\text{m}}}$ (see Figure \ref{domain0}). In this notation, the subscripts ${}_{\text{m}}$ and ${}_{\text{d}}$ stand for ``metal'' and ``dielectric'' respectively. We assume that the real-valued functions $\eps$, $\mu$ verify $\eps:=\eps_r\eps_0$, $\mu:=\mu_r\mu_0$ with
\[
\eps_r = \left\{\begin{array}{ll}\eps_{\text{d}} >0 & \mbox{ in }\Om_{\text{d}}\times\R\\
\eps_{\text{m}}(\omega) <0 & \mbox{ in }\Om_{\text{m}}\times\R\\
\end{array}\right.\qquad \qquad\mbox{ and }\qquad\mu_r = \left\{\begin{array}{ll}\mu_{\text{d}} >0 & \mbox{ in }\Om_{\text{d}}\times\R\\
\mu_{\text{m}}>0 & \mbox{ in }\Om_{\text{m}}\times\R\\
\end{array}\right.,
\]
where $\eps_{\text{m}}(\omega)$ follows the dissipationless Drude's model \eqref{eq:drude}, and $\eps_\di $, $\mu_\di $, $\mu_{\text{m}}$ are three positive constants. Here $\eps_0>0 $ (resp. $\mu_0>0$) refers to the vacuum  permittivity (resp. permeability). If an incident field $u^{\mrm{inc}}  $ independent of the variable $z$ illuminates the obstacle, for instance $u^{\mrm{inc}}(\bfx) = e^{i\bfk \cdot \bfx} $ with $ \vert \bfk \vert = k:=\om\sqrt{\eps_0\mu_0}\sqrt{\eps_\di\mu_\di} $, one can classically reduce the study of Maxwell's equations to the resolution of two uncoupled 2D scalar problems: one associated with $(H_x,H_y,E_z)$ called the Transverse Electric problem (TE), another associated with $(E_x,E_y,H_z)$ called the Transverse Magnetic problem (TM). In particular for the TM problem, $H_z$, denoted by $u$ in the following, is a solution of the problem
\begin{equation}\label{eq:startingpb0}
\begin{array}{|ll}
\text{Find } u = u^{\mrm{inc}} + u^{\mrm{sca}} \text{ such that:} \\[6pt]
\dsp\div\left(\eps_r^{-1}\nabla u \right) + {\om^2}{c^{-2}} \mu_r u = 0 \quad \text{in } \R^2 \\[6pt]
\dsp\lim_{\xi \to + \infty} \int_{\vert \bfx \vert = \xi } \Big\vert \diff{u^{\mrm{sca}}}{r} - i k u^{\mrm{sca}} \Big\vert^2\,d\sigma  = 0,
\end{array}
\end{equation}
where $c := (\sqrt{\eps_0\mu_0})\inv $ denotes the light speed. In (\ref{eq:startingpb0}), the incident field $u^{\mrm{inc}}$ is the data defined above,  the total field $u$ is the unknown and $u^{\mrm{sca}}:=u-u^{\mrm{inc}}$ is the field scattered by the metallic inclusion. The second equation in (\ref{eq:startingpb0}) is the Sommerfeld radiation condition which ensures that $u^{\mrm{sca}}$ is outgoing at infinity ($r$ is the radial coordinate). Here and in the following, we use the same notation for $\eps_r$, $\mu_r$ considered as functions defined on $\R^2$ or $\R^3$. We emphasize that the study of Problem (\ref{eq:startingpb0}) is not standard because $\eps_r $ is sign-changing. 

\begin{figure}[!ht]
\centering
\def\svgwidth{0.62\columnwidth}
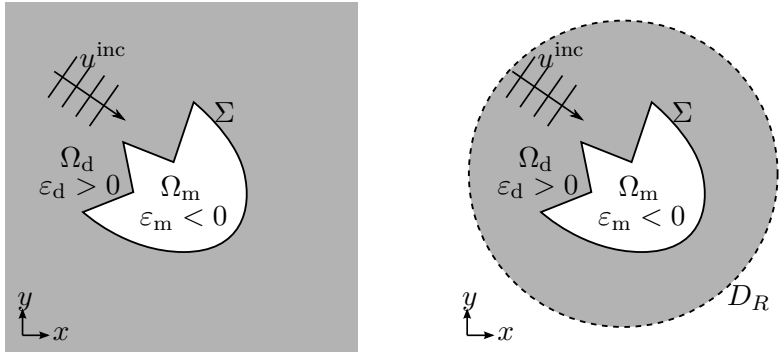
\caption{Left: scattering problem in free space. Right: scattering problem in the artificially bounded domain $D_R$.  \label{domain0}}
\end{figure}

\noindent For the rest of the paper, the index $ _r$ for the parameters $\eps_r $, $\mu_r $ is omitted and we define $k_0 := \omega/c $. We fix the frequency $\omega $ and we write $\eps_\me $ instead of $\eps_\me (\omega) $. Then Problem \eqref{eq:startingpb0} can be rewritten as
\begin{equation}\label{eq:startingpb}
\begin{array}{|ll}
\text{Find } u = u^{\mrm{inc}} + u^{\mrm{sca}} \text{ such that:} \\[6pt]
\dsp\div\left(\eps^{-1}\nabla u \right) + k^2_0 \,\mu u = 0 \quad \text{in } \R^2 \\[6pt]
\dsp\lim_{\xi \to + \infty} \int_{\vert \bfx \vert = \xi } \Big\vert \diff{u^{\mrm{sca}}}{r} - i k u^{\mrm{sca}} \Big\vert^2\,d\sigma  = 0
\end{array}
\end{equation}
or, equivalently, as 
\[
\begin{array}{|rcllrcll}
\multicolumn{8}{|l}{\text{Find } u = u^{\mrm{inc}} + u^{\mrm{sca}} \text{ such that:} \phantom{\eps_{\text{d}}^{-1}\partial_{n}u|_{\Om_{\text{d}}}-\eps_{\text{m}}^{-1}\partial_{n}u|_{\Om_{\text{m}}}  = 0 \quad \text{on } \Sigma \qquad\qquad\qquad }} \\[6pt]
\dsp  \Delta u  + k^2_0 \,\eps_{\text{d}}\mu_{\text{d}}\phantom{.} u & = &  0 \quad \text{in } \Om_{\text{d}} & \qquad\qquad u|_{\Om_{\text{d}}}-u|_{\Om_{\text{m}}} & = & 0 \quad \text{on } \Sigma \\[6pt]
\dsp  \Delta u  + k^2_0 \,\eps_{\text{m}}\mu_{\text{m}} u & = & 0 \quad \text{in } \Om_{\text{m}} & \qquad\qquad \eps_{\text{d}}^{-1}\partial_{n}u|_{\Om_{\text{d}}}-\eps_{\text{m}}^{-1}\partial_{n}u|_{\Om_{\text{m}}} & = & 0 \quad \text{on } \Sigma \\[6pt]
\multicolumn{8}{|l}{\dsp\lim_{\xi \to + \infty} \int_{\vert \bfx \vert = \xi } \Big\vert \diff{u^{\mrm{sca}}}{r} - i k u^{\mrm{sca}} \Big\vert^2\,d\sigma  = 0,}
\end{array}
\]
where $n$ denotes the unit normal to $\Sigma$ oriented from $\Om_{\text{d}}$ to $\Om_{\text{m}}$.

\subsection{Reduction to a bounded domain and description of the geometry}\label{ssec_bounded_dom}

In order to study Problem (\ref{eq:startingpb}), as it is usual in the analysis of scattering problems in free space, we first introduce an equivalent formulation set in a bounded domain. Let $D_R:= \{\bfx\in\eR^2\,|\,|\bfx|< R\}$ denote the open disk of radius $R$. We take $R$ large enough so that $\overline{\Om_{\text{m}}}\subset D_R$. Classically (work e.g. as in \cite[Lemma 5.22]{CaCo06}), one proves that $u $ is a solution of (\ref{eq:startingpb}) if and only if it satisfies the problem 
\begin{equation}\label{eq:Sommerfeld_exacte}
\begin{array}{|lcl}
\div\left({\eps}^{-1}\nabla u\right) + k^2_0 \, \mu u = 0 \quad \text{ in } D_R \\[6pt]
 \diff{u}{r} =   \mathcal{S} u + g^{\mrm{inc}}\quad \text{ on } \partial D_R,\qquad\mbox{ where }g^{\mrm{inc}} :=  \diff{u^{\mrm{inc}}}{r} - \mathcal{S} u^{\mrm{inc}}.
\end{array}
\end{equation}
Here, $(r,\theta)$ stand for the polar coordinates centered at $O$, the center of $ D_R$, while $\mathcal{S}$ refers to the so-called Dirichlet-to-Neumann map. The action of $\mathcal{S}$ can be  described decomposing $u$ in Fourier series (see \cite[Theorem 5.20]{CaCo06}):   
\begin{equation}\label{eq:DtN}
\mathcal{S} u(R,\theta) = \sum \limits_{n= -\infty}^{+\infty}  u_n  \Frac{ k\,H^{(1)'}_n (k R)}{H^{(1)}_n (k R)} \Frac{e^{i n \theta}}{\sqrt{2 \pi}}\qquad\mbox{ with }\qquad u_n = \Frac{1}{\sqrt{2 \pi}} \dsp\int_0^{2 \pi}  u(R,\theta ) e^{-i n \theta}\,d\theta.
\end{equation}
In this expression, $H^{(1)}_n$ denotes the Hankel function of first kind and $H^{(1)'}_n$ its derivative. 

\begin{figure}[!ht]
\centering
\def\svgwidth{0.8\columnwidth}
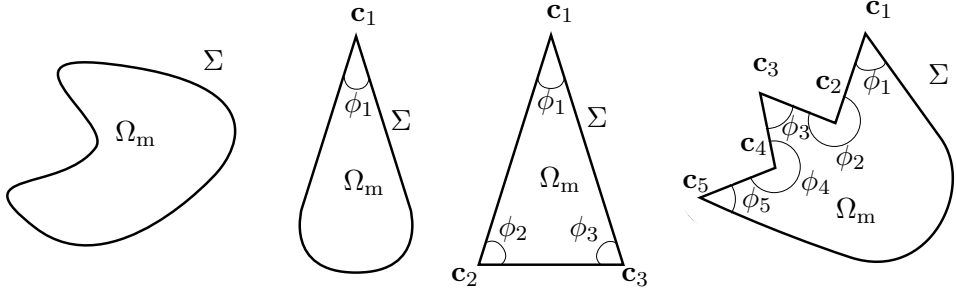
\caption{Examples of geometries. From left to right:  a smooth inclusion ($N=0$), a droplet ($N=1$), a triangle ($N=3$) and a more complicated inclusion ($N=5 $). \label{geometries}}
\end{figure}

We have already mentioned that the geometrical features of the interface $\Sigma=\partial \Om_{\mrm{m}}$ play a key role in the results of well-posedness for problems with sign-changing coefficients. In particular, one has to distinguish when the contrast $\kappa_\eps  := \dsp {\eps_\me}/{\eps_\di}$ is outside or inside the critical interval $I_c$. In order to define this interval, we start by describing the geometry precisely, introducing \textit{ad hoc} notations. We assume that $\Sigma$ is of class $\mathscr{C}^2$ at any $\bfx\in\Sigma$, except at a finite (possibly empty) set of vertices $\bfc_n$, $n=1,\dots,N$. We assume that in a neighbourhood of each vertex $\bfc_n$, $\Om_\text{m}$ coincides with a sector of aperture $\phi_n$ (see an illustration with Figure \ref{geometries}). Set 
\begin{equation}\label{defParamIntervalle}
b_{\Sigma} := \left\{
\begin{array}{ll}
\dsp \max \limits_{n=1,\dots,N} \left( \frac{2\pi -\phi_{n}}{\phi_{n}} ,\frac{\phi_{n}}{2 \pi-\phi_{n}}\right)& \mbox{ if }\Sigma\mbox{ has corners }(N\ge1)\\[6pt]
1 & \mbox{ if }\Sigma\mbox{ is smooth }(N=0).
\end{array}\right.
\end{equation}
\noindent Finally, we define the critical interval $I_c $ by:
\begin{equation}\label{eq:Ic}
I_c := [-b_\Sigma ; -1/ b_\Sigma].
\end{equation}
Details about the derivation of \eqref{eq:Ic} will be given in section \ref{sec:new_framework}. Note that the critical interval $I_c\subset (-\infty;0)$ always contains the value $-1$ (because $b_{\Sigma}\geq 1$) and that there holds $I_c  = \lbrace -1 \rbrace$ if and only if $\Sigma$ is smooth ($N=0$). If $N\ge1$, denoting $\phi_{\min} := \min_{n=1,\dots,N} (\,\phi_{n}\,,\,2\pi -\phi_{n}\,)$, we can write $b_{\Sigma}=(2\pi-\phi_{\min})/\phi_{\min}$. This shows that the critical interval $I_c$ is determined by the aperture of the sharpest angle of the interface. Observe that if $\phi_{\min} \to0$, then $I_c\to (-\infty;0)$.\\

Before studying well-posedness of Problem \eqref{eq:Sommerfeld_exacte}, let us finish this section by some comments. Previous studies for problems with sign-changing coefficients have shown that the adapted functional framework is $H^1$ when the contrast $\ke $ is outside $I_c $, while the solution becomes too singular to belong to $H^1 $ when $\ke \in I_c $ \cite{BoChCi12, BoChCl13}. This will remain true for our scattering problem, and it will have a significant impact on the energy balance. Indeed, when the solution $ u$ of Problem \eqref{eq:Sommerfeld_exacte} is in $H^1(D_R) $, a simple integration by parts yields: \\
\begin{equation}\label{integParts}
\e\di\inv \dsp\int_{\partial D_R} \diff{u}{r}\overline{u}\,d\sigma = \dsp\int_{D_R} \eps^{-1}|\nabla u|^2\,d\bfx - k_0^2\dsp\int_{D_R}  \mu \, |u|^2\,d\bfx.\end{equation}
Taking the imaginary part of (\ref{integParts}), since $\eps$ and $\mu$ are real valued, we obtain the assertion:
 
\begin{equation}\label{eq:DR}
\mbox{ ``$u \in H^1(D_R)$ solution of Problem \eqref{eq:Sommerfeld_exacte}''} \quad \Longrightarrow\quad \Im m\, \left( \dsp \int_{\partial D_R} \frac{\partial u}{\partial r} \overline{u}\, d \sigma \right) = 0.
\end{equation}
This means that the energy flux through $\partial D_R $, in fact through any curve enclosing the obstacle, is equal to $0$ (which seems natural because the medium is non dissipative). We will see in section \ref{sec:new_framework} that this property is not satisfied  when $\kappa_{\eps}\in I_c$. In this situation, some energy is trapped by the corners.

\subsection{Well-posedness in the classical framework for a contrast outside the critical interval}

In this section we explain how to show that Problem \eqref{eq:Sommerfeld_exacte} is well-posed in the usual functional framework $H^1(D_R)$ when $\kappa_{\eps}\notin I_c$. The variational formulation of Problem (\ref{eq:Sommerfeld_exacte}) is given by: 
\begin{equation}\label{eq:FV_Sommerfeld_exacte}
\begin{array}{|lcl}
\text{Find } u \in H^1(D_R) \text{ such that: } \\[6pt]
a(u,v)= l(v), \qquad \forall v \in H^1(D_R),
\end{array} 
\end{equation}
with $a(u,v)\hspace{-0.1cm}=\hspace{-0.1cm} \dsp\int_{D_R} \eps^{-1}\nabla u \cdot \overline{\nabla v}\,d\bfx - k^2_0\hspace{-0.1cm}\dsp\int_{D_R}  \mu \, u \overline{v}\,d\bfx- \sum  \limits_{n= -\infty}^{+\infty}  \frac{k}{\e\di}\,\Frac{ H^{(1)'}_n (kR)}{H^{(1)}_n (kR)} u_n \overline{v_n}$,\quad $l(v)\hspace{-0.1cm}=\hspace{-0.1cm} \dsp\int_{\partial D_R} \frac{g^{\mrm{inc}}}{\e\di} \overline{v}\,d\sigma$.\\[6pt]

\noindent If $u\in H^1_{\mrm{loc}}(\R^2)$ is a solution of the original Problem \eqref{eq:startingpb}, then its restriction to $D_R$ satisfies (\ref{eq:FV_Sommerfeld_exacte}). Conversely if $u$ verifies (\ref{eq:FV_Sommerfeld_exacte}) then it can be extended as a solution of \eqref{eq:startingpb} in $H^1_{\mrm{loc}}(\R^2)$. 
\begin{lemma}\label{lem:uniqueness}
Problem \eqref{eq:startingpb} has at most one solution in $H^1_{\mrm{loc}}(\R^2)$.
\end{lemma}
\begin{proof}
We proceed exactly as when $\eps $ has a constant sign. By linearity of the problem, we just have to show that a solution $u \in H^1_{\mrm{loc}}(\R^2)$ of \eqref{eq:startingpb}  with $u^{\mrm{inc}}=0 $ vanishes. For $\xi \geq R $, there holds
\begin{equation}\label{DvpSommerfeld}
 \dsp\int_{\partial D_\xi } \left|\diff{u}{r}-ik u \right|^2 d\sigma =\dsp\int_{\partial D_\xi } \left|\diff{u}{r}\right|^2 d\sigma+k^2\dsp\int_{\partial D_\xi } \left| u \right|^2 d\sigma- {{2 k\Im m\,(\dsp\int_{\partial D_\xi } \diff{u}{r}\,\overline{u}\,d\sigma)}},
\end{equation}
where $D_\xi =  \lbrace \bfx \in \R^2\,|\,  \vert \bfx \vert < \xi\rbrace$. Then using \eqref{eq:DR} (which is also valid with $R$ replaced by any $\xi \geq R $) and the Sommerfeld radiation condition, we obtain
\[
\lim_{\xi\to+\infty}\dsp\int_{\partial D_\xi } \left| u \right|^2 d\sigma=0.
\]
Since $u$ satisfies $\Delta u + k^2_0 \mu_\di \eps_\di u =0$ in $\R^2\setminus \overline{D_R}$, Rellich's lemma (see e.g. \cite[Theorem 3.5]{CaCo06}) guarantees that $u=0$ in $\R^2\setminus\overline{D_R}$. From Holmgren's theorem \cite[Theorem 8.6.5]{Horm03}, working as in the end of the proof of \cite[Lemma 5.23]{CaCo06}, successively we deduce that $u=0$ in $\Om_{\text{d}}$ and $u=0$ in $\Om_{\text{m}}$.
\end{proof}

Lemma \ref{lem:uniqueness} ensures in particular that a metallic inclusion in the free space cannot trap a pure resonant wave. Now let us prove the main result for Problem (\ref{eq:startingpb}) outside the critical interval.

\begin{proposition}\label{prop:pb_wp}
Let $\om>0$ be a given frequency. Assume that the contrast $\kappa_{\eps}=\e{\text{m}}/\e{\text{d}}$ verifies $\kappa_{\eps}\notin I_c=[-b_{\Sigma};-1/b_{\Sigma}]$, where $b_\Sigma$ has been defined in \eqref{defParamIntervalle}. Then Problem (\ref{eq:startingpb}) has a unique solution $u$ in $H^1_{\mrm{loc}}(\R^2)$. Moreover there exists $C>0$ independent of the data $g^{\mrm{inc}} $ such that
\[ \Vert u \Vert_{H^1(D_R)} \leq C \Vert g^{\mrm{inc}}  \Vert_{L^2(\partial D_R)}.\]
\end{proposition}
\begin{proof}Let us decompose the bilinear form $a(\cdot,\cdot)$ of the variational formulation \eqref{eq:FV_Sommerfeld_exacte} as follows:
\begin{equation}\label{bilinear_b}
\begin{array}{ll}
&\forall u,v \in H^1(D_R), \quad  a(u,v)  = b(u,v) + c(u,v),\\
& \mbox{with } \quad b(u,v)=  \dsp\int_{D_R} \eps^{-1}\nabla u \cdot \overline{\nabla v}\,d\bfx + \e\di\inv  \sum \limits_{n= -\infty}^{+\infty} \Frac{ n}{R}\, u_n  \overline{v_n },
\end{array}
\end{equation}
$u_n $, $v_n$ being defined as in \eqref{eq:DtN} (one deduces $c(\cdot, \cdot)$ from above). With the Riesz representation theorem, we introduce the maps $\mathcal{B}:H^1(D_R)\to H^1(D_R)$, $\mathcal{C}:H^1(D_R)\to H^1(D_R)$ as well as the function $f \in H^1(D_R)$ such that
\begin{equation} \label{def_op}
(\mathcal{B} u,v)_{H^1(D_R)}  = b(u,v),\quad (\mathcal{C} u,v)_{H^1(D_R)}= c(u,v), \quad (f,v)_{H^1(D_R)} = l(v), \quad \forall u,v \in H^1(D_R). 
\end{equation} 
With these notations, $u$ is a solution of \eqref{eq:FV_Sommerfeld_exacte} if and only if it satisfies $\mathcal{B}u +\mathcal{C}u = f $. Using classical results concerning the Dirichlet-to-Neumann map $\mathcal{S}$ (following e.g. \cite[Theorem 5.20]{CaCo06}), as well as the compact embedding of $H^1(D_R)$ in $L^2(D_R)$, one can prove that $\mathcal{C}$ is a compact operator. Therefore it is sufficient to show that $\mathcal{B}$ is a Fredholm operator to conclude. \\
In the classical case of a positive $\eps$, one obtains easily the result thanks to the Lax-Milgram theorem. Here, because of the change of sign of $\eps$ in $D_R$, $b(\cdot, \cdot) $ is not of the form ``coercive+compact'' and this technique fails. Instead, we use the $\tT $-coercivity approach. More precisely, when $\ke \not \in  I_c$, we  prove in Annex \ref{Annex-Tcoercivity} the existence of a bounded operator $\tT:H^1(D_R)\to H^1(D_R)$ such that
\begin{equation}\label{eq:T_coerc}
\mathcal{B}\circ\tT = \mathcal{I}+\mathcal{K},
\end{equation}
where $\mathcal{I}:H^1(D_R)\to H^1(D_R)$ is an isomorphism and where $\mathcal{K}:H^1(D_R)\to H^1(D_R)$ is compact. Since $\mathcal{B}$ is selfadjoint ($\mathcal{B}$ is bounded and symmetric), according to classical results of functional analysis (see \cite[Theorem 12.5]{Wlok87}), this is enough to conclude that it is a Fredholm operator (of index zero).
\end{proof}
\noindent Let us make some comments regarding this approach:\\[5pt]
$\bullet\ $ The analysis developed here can be adapted to varying coefficients $\eps,\, \mu $. In this case, the condition \cite[Theorem 4.3]{BoChCi12} to guarantee Fredholmness for Problem (\ref{eq:Sommerfeld_exacte}) with the \texttt{T}-coercivity technique involves the ratio of local upper and lower bounds of $\eps_{\text{d}}$, $\e{\text{m}} $ in a neighbourhood (which can be chosen as small as we want) of $\Sigma$. In order to prove uniqueness of the solution, an additional smoothness assumption on $\eps$ (imposing e.g. $\eps$ to be piecewise smooth) has to be made to apply unique continuation results.\\
$\bullet\ $ In the (TE) problem, the electric field $E_z$ satisfies the equation $\div\left(\mu^{-1}\nabla E_z \right) + k_0^2 \eps E_z = 0$ in $\R^2$. Thus, since $\mu_{\text{m}}>0$, we find that the (TE) problem is always well-posed: the change of sign of $\eps$ does not matter in the (TE) problem. \\
$\bullet\ $ One can also consider sign-changing permeabilities, for example when one wishes to model the propagation of electromagnetic fields in presence of Negative Index Metamaterials (NIM) \cite{Pend00,SmPW04,Anan05}. Defining the contrast $\kappa_{\mu} := \mu_{\text{m}}/\mu_{\text{d}}$, for the (TM) problem, well-posedness holds for $\kappa_{\eps}\notin I_c$ and for all $\kappa_{\mu}<0$. Also for the (TE) problem, well-posedness is guaranteed for all $\kappa_{\eps}<0$ when $\kappa_{\mu}\notin I_c$.

\section{Singular solutions at one corner}\label{sec:new_framework}

When the interface between the two materials presents corners, then for a contrast $\kappa_{\eps}$ inside the critical interval $I_c$, well-posedness of Problem (\ref{eq:startingpb}) in the usual $H^1$ framework is lost. This is due to the appearance of strongly oscillating singularities at one or several corner(s). This has been already investigated in \cite{DaTe97,BoChCl13,ChCNSu} for a particular geometry involving one corner of aperture $\pi/2 $. Here we wish to study the general case of an arbitrary corner angle.

\subsection{Characterization of singular exponents}\label{ssec:singular_expo}

For ease of exposition, we consider a metallic inclusion with only one corner, that we denote by $\bfc$. Without loss of generality we assume that $\bfc $ is located at the origin $O$. In accordance with the setting of \S\ref{ssec_bounded_dom}, $\Om_\text{m}$ coincides in the vicinity of the vertex $\bfc $ with a sector of aperture $\phi\in (0; 2 \pi)\setminus\{\pi\}$ (see Figure \ref{img:zoom} left, page \pageref{img:zoom}). In other words, there exists $\rho>0 $ such that for well-chosen cartesian and polar coordinates $(r,\theta)$: 
\[
\begin{array}{lcl}
D_{\rho}\cap \Om_\text{d} & \hspace{-0.1cm}=\hspace{-0.1cm}  & \{\bfx=(r\cos(\theta),r\sin(\theta))\,|\,0<r<\rho,\,\phi/2 < \vert \theta \vert  < \pi \}, \\[2pt]
D_{\rho}\cap \Om_\text{m} & \hspace{-0.1cm}=\hspace{-0.1cm}  &  \{\bfx=(r\cos(\theta),r\sin(\theta))\,|\,0<r<\rho,\,\vert \theta \vert < \phi/2\},
\end{array}
\]
 where $D_\rho$ denotes the disk of radius $\rho $. In $D_\rho $, the permittivity $\eps $ depends only on $\theta $: 
\[
\eps = \left\{
\begin{array}{ll} 
\eps_\di & \mbox{for }  \phi/2 < \vert \theta \vert  < \pi,\\[3pt]
\eps_\me & \mbox{for } \vert \theta \vert < \phi/2. 
  \end{array}\right. 
\]
 \noindent
In polar coordinates, equation of Problem \eqref{eq:startingpb} multiplied by $r^2 $ reads
 \begin{equation}\label{eq:pb_polar_coord}
 \dsp  \eps\inv (r \partial_r)^2  u  + \partial_\theta (\eps\inv \partial_\theta u) + k^2_0 r^2  \mu u = 0,
 \end{equation}
where we use abusively the same notation for $u(\bfx)$ and $u(r,\theta) $. Zooming at the corner (\textit{i.e.} taking $\rho $ small enough) leads to neglect the term $k^2_0 r^2 \mu u$ in (\ref{eq:pb_polar_coord}) and to study the ``static'' equation \eqref{EquationChangeOfSign}. In other words, the singular behaviour of a solution of \eqref{eq:pb_polar_coord} is the same as the one of a solution of \eqref{EquationChangeOfSign}. The main advantage of considering \eqref{EquationChangeOfSign} is that separation of variables is now possible. In the following, we call singularities the functions $s(r,\theta)=\chi(r)\Phi(\theta)$ with separated variables in polar coordinates which satisfy
\begin{equation}\label{pbVariablesSeparees}
\div(\eps^{-1}\nabla s)=0\qquad \mbox{in } D_\rho,
\end{equation}
that is $\dsp  \eps\inv (r \partial_r)^2  s  + \partial_\theta (\eps\inv \partial_\theta s)=0 $. A direct calculation shows that this holds if and only if there exists $\lambda\in \mathbb{C} $ such that $(r\partial_r)^2\chi=\lambda^2\,\chi$, for $0<r <\rho$ and
\begin{equation}\label{pb_phi}
\begin{array}{|rcl}
-(\partial_{\theta}\eps^{-1} \partial_{\theta})\Phi &= &\lambda^2\,\eps^{-1}\Phi \quad  \mbox{for }- \pi < \theta < \pi,\\[2pt]
\Phi(-\pi) &= & \Phi(\pi),\\[3pt] 
\partial_\theta \Phi (-\pi) & = &\partial_\theta \Phi (\pi).
\end{array}
\end{equation}
The problem verified by $\chi$, for which we do not impose boundary conditions, can be easily solved. For $\lambda=0$, we find $\chi(r)=A\,\ln r+B$ whereas for $\lambda\ne0$, we obtain $\chi(r)=A\,r^\lambda+B\,r^{-\lambda}$, $A$, $B$ being two constants. The Problem \eqref{pb_phi} satisfied by $\Phi$ is an eigenvalue problem. Denote by $\Lambda$ the set of values $\lambda$ such that Problem \eqref{pb_phi} has a non zero solution $\Phi $. This set will be referred to as the set of singular exponents associated to $\bfc$. In the usual case where $\eps>0 $, (\ref{pb_phi}) is a self-adjoint problem with positive eigenvalues $\lambda^2 $ ($\Lambda \subset\R$). Here, because the sign-changing parameter $\eps$ appears in both sides of the equation, the study and the properties of \eqref{pb_phi} are not standard. In particular, $\Lambda $ generally contains complex eigenvalues. Concerning the analysis, to our knowledge there is no theory for this eigenvalue problem. However since it is a $1$D problem, we can carry out explicit computations. First we can prove that, when $\kappa_{\eps}\ne-1$, $\Lambda$ is discrete (proceed as in \cite[Corollary 4.10]{ChCi11}). Moreover, it is straightforward to see that $0 \in \Lambda $ and that $ \Lambda$ is stable by symmetry/conjugation ($  -\lambda, \, \pm \overline{\lambda}  \in \Lambda$ for all $ \lambda \in \Lambda$). Using elementary symmetry arguments, one can compute singular exponents associated with even or odd eigenfunctions $\Phi $. The corresponding singularities $s(r, \theta)$ are either symmetric or skew-symmetric with respect to the bisector of the corner $\ell := \lbrace \bfx =(r \cos \theta, r \sin \theta) \in \R^2\,|\,\, \theta =0 \rbrace $ (see Figure \ref{img:zoom} left). We denote by $\Ls $ (resp. $\Lss $) the subset of $\Lambda $ of singular exponents corresponding to symmetric (resp. skew-symmetric) singularities. Note that $\Lambda = \Ls \cup \Lss $. Simple calculations reproduced in Annex \ref{annex:compute} yield the following characterizations for the sets $ \Ls$, $\Lss $. 

\begin{proposition}\label{prop:fonction_zeros}
Set $b:= ({2 \pi - \phi})/{\phi} $ and $f^\pm (z) = \kappa_\eps^{\pm 1} \tanh(z) + \tanh (bz)$ for $z \in \eC$. We have 
\[
\Ls :=\lbrace\lambda \in \eC\,|\, \dsp f^- \left(i \lambda \phi/{2} \right) = 0\rbrace, \qquad \quad \Lss :=\lbrace\lambda \in \eC\,| \, \dsp f^+ \left(i \lambda {\phi}/{2} \right) = 0\rbrace.
\]
\end{proposition}
 All singularities do not contribute to the solution of a problem such as \eqref{eq:startingpb}. For instance, for $\kappa_{\eps}\notin I_c $, the solution of \eqref{eq:startingpb} is locally in $H^1$, so that only singularities which are in $H^1$ near the corner should be considered. For $\lambda = 0 $, $A \ln r +B $ is locally in $H^1 $ if and only if $A=0$. For $\lambda $ such that $\Re e \, \lambda >0$, $s(r,\theta) = r^\lambda \Phi(\theta) $ is locally in $H^1 $ while $s(r,\theta) = r^{-\lambda} \Phi(\theta) $ is not. The singularities associated with singular exponents $\lambda$ satisfying $\lambda\ne0$ and $\Re e\,\lambda=0$ (located at the limit between the $H^1$ zone and the non $H^1$ zone) play a special role for Problem \eqref{eq:startingpb}. We will have to take them into account in the functional framework even though they do not belong to $H^1$ (see (\ref{petitCalcul})). In the next paragraph, we focus our attention on these particular singularities.
 
 \subsection{Oscillating singularities}
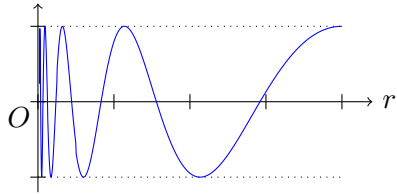
\begin{figure}[h!]
\centering
\begin{tikzpicture}
\draw[->] (0,0) -- (4.4,0) node [right] {$r$};
\draw (-0.25,-0.2) node {$O$};
\draw[dotted] (0,1) -- (4,1);
\draw[dotted] (0,-1) -- (4,-1);
\draw[samples=800,domain=0.02:4,blue] plot(\x,{cos((deg(ln(\x*0.25)*5)))});
\draw[->] (0,-1.2) -- (0,1.3);
\draw (-0.3,1);
\draw (0.3,-1);
\foreach \x in {0,...,4} \draw (\x,0.1) -- (\x,-0.1);
\foreach \y in {-1,...,1} \draw (-0.1,\y) -- (0.1,\y);
\end{tikzpicture}
\caption{Behaviour of the real part of the radial component of the singularity $r^{\lambda}\Phi(\theta)$, for $\lambda\in\R i\setminus\{0\}$, in a neighbourhood of $O$. To understand these oscillations, observe that $\Re e\,r^{i\eta}=\cos(\eta\ln r)$ for $\eta\in\R^{\ast}$.\label{dessin propagative singu coupe 1D}}
\end{figure}

Assume that $\Lambda \cap i \R $ contains some $\lambda \neq 0$. The singularities $r^{\pm\lambda} \Phi(\theta)$ have a curious oscillating behaviour at the origin (see Figures \ref{dessin propagative singu coupe 1D}, \ref{img:black-hole_waves}) and do not belong to $H^1 $. Indeed, for $\lambda =i \eta$, $ \eta \in \eR^\ast $, we obtain
\begin{equation}\label{petitCalcul}
\lim_{\delta \rightarrow 0^+}\ \int_{-\pi}^{\pi}\int_{\delta}^{\rho}\vert \partial_r (r^{i\eta} \Phi(\theta)) \vert^2 \, rdr d\theta = \lim_{\delta \rightarrow 0}\  \eta^2 \dsp \,  \int_{-\pi}^{\pi} \vert \Phi(\theta) \vert^2 \, d\theta  \, \int_{\delta}^{\rho}  \, \frac{dr}{r}  = +\infty.
\end{equation}
Going back to Maxwell's equations, this means that the electric field $\boldsymbol{E}$ is such that $\boldsymbol{E} \not \in (L^2_{\mrm{loc}}(\R^3))^3$, so that the energy is infinite at the corner. Now, we prove that such singularities exist. 
\begin{lemma}\label{lem:exist_ilambda}
Define $I_c = [-b_\Sigma ; -1/b_\Sigma] $ with $b_\Sigma = \max \left( \frac{2\pi - \phi}{\phi},\frac{ \phi}{2\pi -\phi} \right) $. The following array describes the set $\Lambda \cap i \eR$ with respect to $\kappa_{\eps}$, $\phi$:
\[
\renewcommand{\arraystretch}{1.5}
\begin{array}{>{\centering}m{1.7cm}|c|c|c|}
\cline{2-4} 
& \kappa_\eps \not \in I_c & \kappa_\eps \in (-b_\Sigma; -1 ) & \kappa_\eps \in (-1 ; -1/b_\Sigma ) \\ \hline 
\multicolumn{1}{|c|}{0<\phi<\pi} & \Lambda \cap i \eR = \lbrace 0 \rbrace\hspace{-0.2cm}~ &
\hspace{-0.2cm}\begin{array}{l}
\Lambda \cap i \eR = \lbrace \pm i \eta\rbrace,\,\mbox{ for some } \eta >0\\[-1pt]
\mbox{ and}\quad\bigg\{\begin{array}{ccl}
\Ls \cap i \eR &=& \{0\} \\[-4pt]
\Lss \cap i \eR &=& \lbrace \pm i \eta\rbrace
\end{array}
\end{array} \hspace{-0.35cm}~& \hspace{-0.15cm}\begin{array}{l}
\Lambda \cap i \eR = \lbrace \pm i \eta\rbrace,\,\mbox{ for some } \eta >0\\[-1pt]
\mbox{ and}\quad\bigg\{\begin{array}{ccl}
\Ls \cap i \eR &=& \lbrace \pm i \eta\rbrace \\[-4pt]
\Lss \cap i \eR &=& \{0\}
\end{array}
\end{array}\hspace{-0.4cm}~ \\[6pt] \hline
\multicolumn{1}{|c|}{\pi<\phi<2\pi} & \Lambda \cap i \eR = \lbrace 0 \rbrace\hspace{-0.2cm}~ &\hspace{-0.2cm}
\begin{array}{l}
\Lambda \cap i \eR = \lbrace \pm i \eta\rbrace,\,\mbox{ for some } \eta >0\\[-1pt]
\mbox{ and}\quad\bigg\{\begin{array}{ccl}
\Ls \cap i \eR &=& \lbrace \pm i \eta\rbrace \\[-4pt]
\Lss \cap i \eR &=& \{0\}
\end{array}
\end{array} \hspace{-0.35cm}~&\hspace{-0.15cm} \begin{array}{l}
\Lambda \cap i \eR = \lbrace \pm i \eta\rbrace,\,\mbox{ for some } \eta >0\\[-1pt]
\mbox{ and}\quad\bigg\{\begin{array}{ccl}
\Ls \cap i \eR &=& \{0\} \\[-4pt]
\Lss \cap i \eR &=& \lbrace \pm i \eta\rbrace
\end{array}
\end{array}\hspace{-0.4cm}~ \\[6pt] \hline
\end{array}
\]
\end{lemma}
\begin{proof}
 According to Proposition \ref{prop:fonction_zeros}, the singular exponents are given by the zeros of the functions $f^\pm $. As the functions $f^\pm $ are odd, it is sufficient to study their zeros on $(0; +\infty)$.\\[2pt]
$\star$ First assume that $0<\phi<\pi$. Then, we have $b = \dsp ({2\pi -\phi})/{\phi} >1 $ and $b_\Sigma = b $. We can check that $f^\pm $ do not vanish when $ \ke \not \in [-b_\Sigma; -1/b_\Sigma]$:\\
- if $\ke < -b_\Sigma $, then on $(0;+\infty) $, $f^-(t) > (  1+ \kappa_\eps^{-1} )\tanh(t)>0 $ and $f^+(t)< \tanh(b t) -b \tanh(t) <0$;\\
- if $\ke >-1/ b_\Sigma $ then on $(0;+\infty) $, $f^-(t) < \tanh(b t) -b \tanh(t) <0 $ and $f^+(t)>  (  1+ \kappa_\eps)\tanh(t)>0$.\\[2pt]
Now, we wish to prove that \\[5pt]
$\begin{aligned}
&\kappa_\eps \in (-b_\Sigma; -1 ) \quad &\Longrightarrow \quad \exists ! \, \eta >0 \mbox{ such that } f^+\left(\eta {\phi}/{2} \right) = 0 \mbox{ and } f^-(t) \neq 0, \, \forall t \in (0;+\infty),\\
&\kappa_\eps \in (-1; -1/b_\Sigma ) \quad &\Longrightarrow \quad \exists ! \, \eta >0 \mbox{ such that } f^-\left(\eta {\phi}/{2} \right) = 0 \mbox{ and } f^+(t) \neq 0, \, \forall t \in (0;+\infty).
\end{aligned}$\\[5pt]
$\bullet$ Case $\kappa_\eps\in(-b_\Sigma;-1)$. With the same arguments as before we have $f^{-}>0$ on $(0;+\infty)$. On the other hand, a careful analysis of the monotony shows that $f^{+}$ vanishes exactly once on $(0;+\infty)$ if and only if $\kappa_\eps\in(-b_\Sigma;-1)$. More precisely, one checks that $(f^+)'(0)=\ke+b_\Sigma>0 $ (and $f^+(0)=0$), while $\lim_{t \rightarrow+\infty}f^+(t) =\ke+1<0 $. This proves that $f^+$ has at least one zero in $(0;+\infty)$. Then one proves that the derivative of $f^+ $ changes sign once and only once on $(0;+\infty) $ to conclude. \\[2pt]
$\bullet$ Case $\kappa_\eps\in(-1:-1/b_\Sigma)$. With analogous  arguments we obtain $f^{+}>0$ on $(0;+\infty)$ and we establish that $f^{-}$ has exactly one zero on $(0;+\infty)$ if and only if $\kappa_\eps\in(-1:-1/b_\Sigma)$. \\[5pt]
$\star$ To consider the situation $\pi <\phi<2\pi$, it is sufficient to note that the singularities of the operators $\div(\eps^{-1}\nabla\cdot)$ and $-\div(\eps^{-1}\nabla\cdot)$ are the same and to use the results of the case $0<\phi<\pi$ with $\kappa_{\eps}$ replaced by $1/\kappa_{\eps}$. Indeed, with this multiplication by $-1$, the roles of $\Om_{\mrm{m}}$ and $\Om_{\mrm{d}}$ are exchanged.
\end{proof}
\begin{figure}[!h]
\centering
\includegraphics[width=0.27\columnwidth]{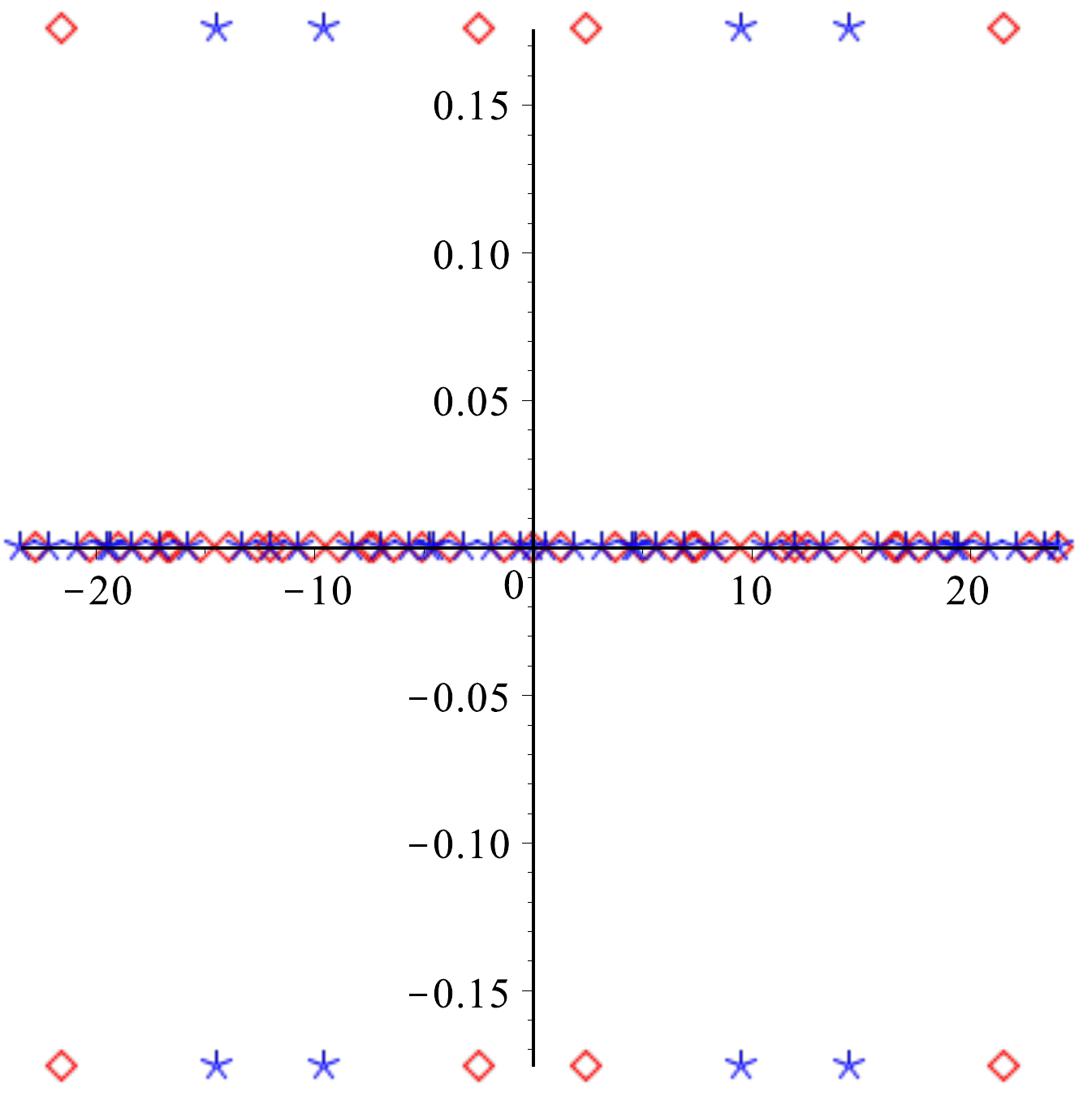}\quad
\includegraphics[width=0.27\columnwidth]{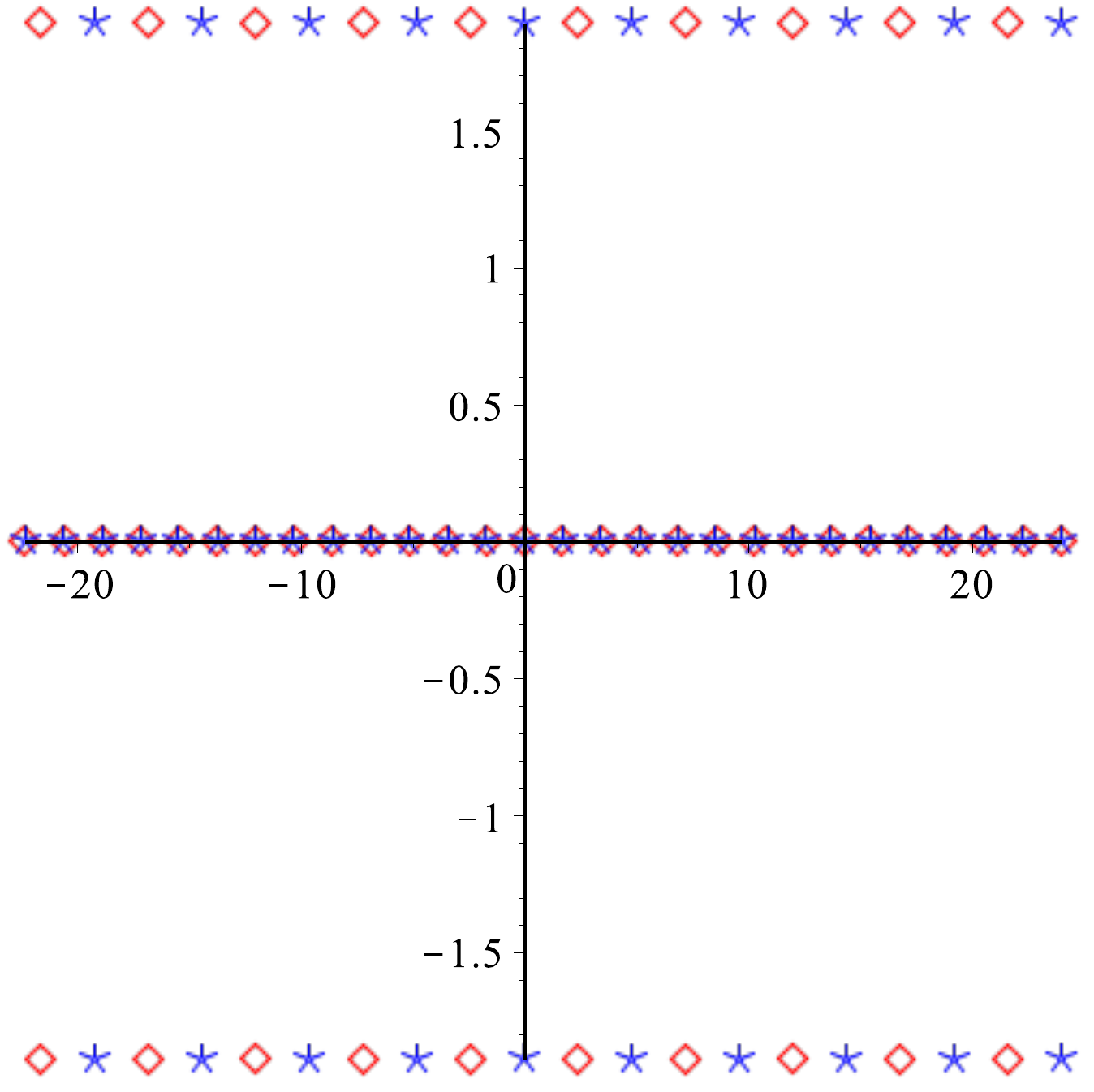}\quad
\includegraphics[width=0.27\columnwidth]{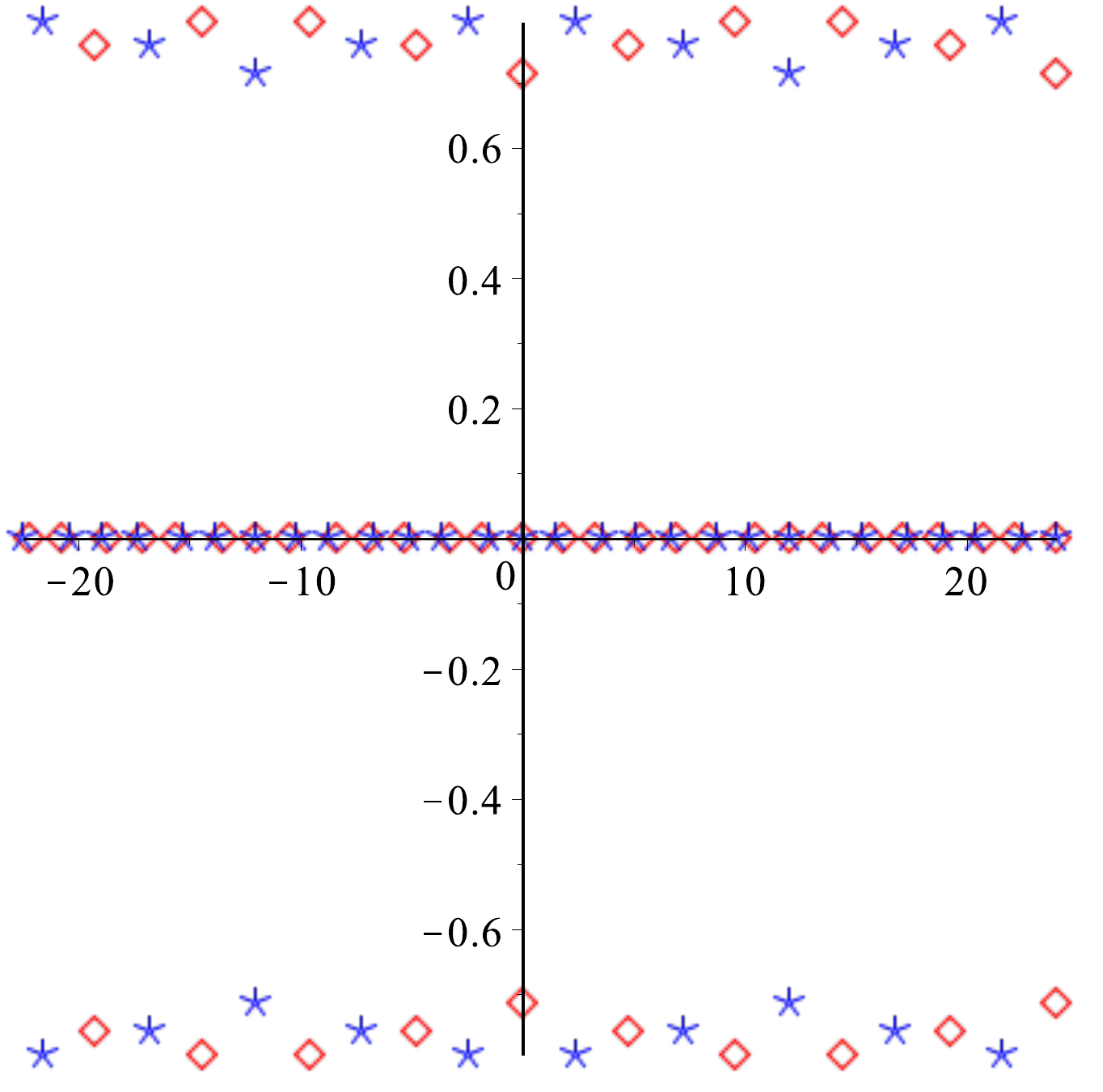}
\caption{Set of singular exponents $\Lambda$ for $\phi = 5\pi/12 $ (in this case $I_c =[-3.8; -0.26315]$). The singular exponents $\lambda \in \Lss $ are represented with asterisks while the singular exponents $\lambda \in \Ls $ are represented with diamonds. Left: $ \ke=-18.684$. Middle: $ \ke=-1.1871$. Right: $ \ke=-0.4641$.\label{img:spectre_compa}}
\end{figure}
In Figure \ref{img:spectre_compa}, we display the set of singular exponents $\Lambda$ for contrasts outside and inside the critical interval $I_c$ for a given angle $\phi$. The results are in accordance with Lemma \ref{lem:exist_ilambda} (observe also that $\Lambda $ may contain complex eigenvalues even if $\ke\notin I_c$). In the following, for a contrast $\kappa_{\eps}\in (-b_{\Sigma};-1/b_{\Sigma})\setminus\{-1\}$, we shall often refer to $s^{\pm}$ where 
\[
s^{\pm}(r,\theta)=r^{\pm i\eta}\Phi(\theta).
\]
We recall that $\eta$ is chosen positive. With the results of Proposition \ref{prop:fonction_zeros} and Lemma \ref{lem:exist_ilambda} one can check that the eigenfunctions $\Phi $ are defined as follows.
\begin{equation}\label{PhiSkewSym}
\begin{array}{l|l}
\hspace{-0.9cm}\mbox{If }\Lss \cap i \eR\ne\{0\}, & \dsp\Phi(\theta) = \frac{\sinh(\eta\theta)}{\sinh(\eta\phi/2)}\ \mbox{on}\ [0;\phi/2];\  \Phi(\theta) = \frac{\sinh(\eta(\pi-\theta))}{\sinh(\eta(\pi-\phi/2))}\ \mbox{on}\ [\phi/2;\pi];\\[10pt]
 & \Phi(\theta)=-\Phi(-\theta)\ \mbox{on}\quad[-\pi;0].
\end{array}
\end{equation}
\begin{equation}\label{PhiSym}
\begin{array}{l|l}
\hspace{-0.9cm}\mbox{If }\Ls \cap i \eR\ne\{0\}, & \dsp\Phi(\theta) = \frac{\cosh(\eta\theta)}{\cosh(\eta\phi/2)}\ 
\mbox{on}\ [0;\phi/2];\  \Phi(\theta) = \frac{\cosh(\eta(\pi-\theta))}{\cosh(\eta(\pi-\phi/2))}\ \mbox{on}\ [\phi/2;\pi];\\[10pt]
& \Phi(\theta)=\Phi(-\theta)\quad\mbox{on}\quad[-\pi;0].
\end{array}
\end{equation}

\begin{figure}[!h]
\centering
\includegraphics[width=0.25\columnwidth]{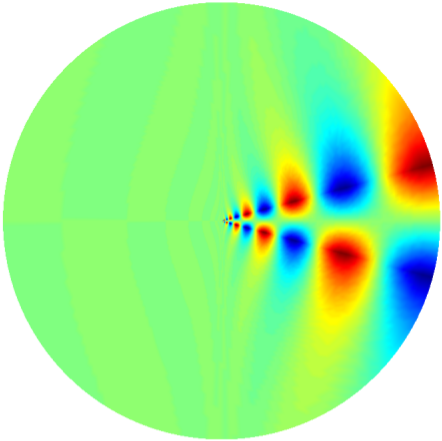} \qquad
\quad\includegraphics[width=0.25\columnwidth]{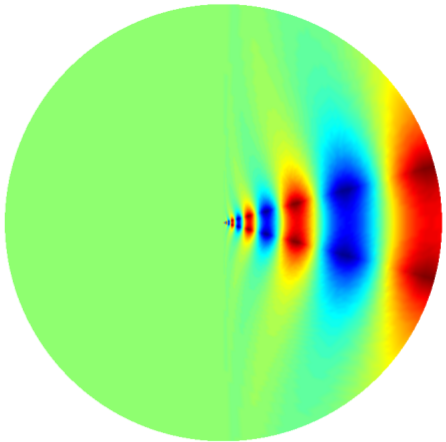} 
\caption{Real part of the skew-symmetric (left) and symmetric (right) oscillating singularity near the corner $\bfc$.\label{img:black-hole_waves}}
\end{figure}

Going back to the time-domain and multiplying the oscillating singularities by the harmonic term $e^{-i\om t}$ leads us to consider functions which behave like 
\begin{equation}\label{WavesBidule}
s^{\pm}(r,\theta)e^{-i\omega t} = e^{i(\pm\eta\ln r-\omega t)}\Phi(\theta)
\end{equation}
near the corner $\bfc$. Let us compute the phase velocity of these waves. A point $M$ of phase $i(\pm\eta\ln r-\omega t)$ will be located at $r+dr$ at $t+dt$ with $\pm\eta\ln(r+dr)-\omega (t+dt)=\pm\eta\ln r-\omega t$. Since $\ln (r+dr)  = \ln r + dr/r + o(1/r)$ for $dr$ small, one finds the phase velocity $dr/dt = \pm \omega r/\eta $ for the waves (\ref{WavesBidule}). Note that it tends to zero when approaching the origin. The wave $s^{-}(r,\theta)e^{-i\omega t}$ seems to propagate to the corner but never reaches it. This is the reason why in the following, it will be referred to as ``black-hole wave''. By extension, we will use the same denomination for the other wave $s^{+} (r,\theta)e^{-i\omega t}$ which seems to come from $\bfc$. Finally, we point out that from time to time in this paper, the oscillating singularities will be called ``black-hole singularities''.

\subsection{Selecting the outgoing solution I: energy trapped at the corner} \label{sec_outgoing}
For a contrast $\kappa_\eps \in I_c\setminus \lbrace -1 \rbrace $, looking for a solution of the scattering problem in $H^1_{\mrm{loc}}$ (\textit{i.e.} with a local finite electromagnetic energy) would lead to exclude a behaviour at a corner like the one of the oscillating singularities $s^\pm (r,\theta)$. Nonetheless allowing such singular behaviour is necessary to obtain existence (and uniqueness) of a solution $u$ to Problem \eqref{eq:startingpb}. This is exactly as in waveguides problems for which we must look for  solutions which decompose on propagative modes which are not in $H^1$. Actually, there is a strong analogy between the problem considered here and waveguides problems which is described in \S\ref{paragraphAnalogyWaveguide}. Thus, we are led to consider $u$ which decomposes, in a neighbourhood of $\bfc$, as 
\begin{equation}\label{firstExpansion}
u= a+b_+ s^{+} +b_- s^{-}+\tilde{u},
\end{equation}
where $\tilde{u}$ is a smooth function and $a,\,b_{\pm}$ are complex constants. Roughly speaking $\tilde{u}$ is smooth means that it is a superposition of singularities with singular exponents  $\lambda $ such that $\Re e\,\lambda >0 $ (see \S\ref{proof_lemma_energy_DIC}). We recall that $0$ belongs to $\Lambda$ for all $\kappa_{\eps}<0$, that the associated singularities are $1$ and $\ln r$, and that only the constant is locally in $H^1$. In order to get uniqueness of the solution for \eqref{eq:Sommerfeld_exacte}, as shown in \cite{BoChCl13}, a relation on $b_{\pm}$ has to be enforced. A priori, it is not obvious to decide which condition to impose that will give the ``physical'' solution because the singularities $s^{\pm}$ have very similar behaviours at $\bfc$. In particular, due to the change of sign of the permittivity $\eps$, considerations based on phase velocity are not sufficient. To identify the relevant condition, we study in this paragraph the energy carried by $s^{\pm}$. Let us consider a function $ u$ of the form \eqref{firstExpansion} which satisfies the equation $\div (\eps \inv \grad u)+ k^2_0 \mu u =0 $ in the vicinity of the corner $ \bfc$. Proceeding like at the end of \S \ref{ssec_bounded_dom}, one can easily verify that the quantity 
\begin{equation}\label{def_J}
J := \dsp \Im m \left( \int_{\partial D_\rho} \eps\inv \partial_r u \overline{u} \, d \sigma\right) 
\end{equation} defined for $\rho $ small enough, is independent of $\rho $. It represents the energy flux through $\partial D_\rho $ coming from the corner. Plugging \eqref{firstExpansion} in the left-hand side of (\ref{def_J}) yields
\begin{equation}\label{J}
J= \eta (\vert b_+ \vert^2 - \vert b_- \vert^2) \dsp \int_{-\pi}^{\pi} \eps\inv \Phi^2 \, d\theta.
\end{equation}
Indeed, on the one hand, using that $\Phi $, given by \eqref{PhiSkewSym} or \eqref{PhiSym}, is real-valued, one can check that
\begin{equation}\label{CalculPlusMoins}
 \dsp\int_{\partial D_\rho } \eps^{-1}\diff{s^\pm}{r}\,\overline{s^\pm}\,d\sigma = \pm   i\eta \dsp\int_{-\pi}^{\pi}\eps^{-1} \Phi^2\,d\theta.
\end{equation}
On the other hand, one can prove that all the other cross-terms tend to $0$ as $ \rho \rightarrow 0$ (proceed as in \S\ref{proof_lemma_energy_DIC}). Then identity \eqref{J} follows by noting that $\overline{s^+} =s^- $. The sign of the integral appearing in (\ref{J}), which is not obvious because of the presence of the parameter $\eps$, plays an important role to compute energy balances. 
Explicit calculations detailed in the Annex (see Lemma \ref{ResultatCalcul}) show that
\begin{equation}\label{relationSpectral}
\int_{-\pi}^{\pi}\eps^{-1} \Phi ^2\,d\theta>0\quad\mbox{ if}\ \kappa_\eps \in (-b_\Sigma; -1 )\quad \mbox{ and }\quad \int_{-\pi}^{\pi}\eps^{-1} \Phi ^2\,d\theta<0\quad\mbox{ if}\ \kappa_\eps \in (-1;-1/b_\Sigma).
\end{equation}
We see that the sign of the integral depends only on the contrast of the physical parameters of the two materials. If $\eps^{-1}_\di>|\eps_\me|^{-1}$, then the integral is positive, and vice versa.\\
\newline
Now, consider for instance the case $\kappa_\eps \in (-b_\Sigma; -1 )$. Since by definition $\eta $ is positive, we observe with \eqref{J} that the singularity $s^+ $ adds a positive contribution to the energy flux $J$. It means that $s^+$ carries some energy produced by the corner. We say that $s^+$ is the ``incoming'' singularity ($s^\inc=s^+ $) because it brings energy into the system. On the contrary, $s^-$ adds a negative contribution to $J$: it carries some energy absorbed by the corner. We say that $s^-$ is the ``outgoing'' singularity ($s^\out=s^- $). When $\kappa_{\eps}\in(-1;-1/b_{\Sigma})$, according to (\ref{relationSpectral}), we take $s^\out=s^+$ and $s^\inc=s^-$. The results are summarized in Table \ref{tableSummary} below. To conclude, in the following, we impose that the solution of Problem \eqref{eq:startingpb} decomposes  only on the outgoing singularity $s^\out $ and not on $s^\inc$ because $s^\inc $ carries some energy produced by the corner which is not physical.

\begin{remark}
The terminology ``incoming/outgoing'', inspired by the scattering theory, is mainly related to the point of view developed in section \ref{sec:PMLs}, where $s^\out $ (resp. $s^\inc $) corresponds to an outgoing (resp. incoming) propagative mode in a waveguide. 
\end{remark}

\begin{table}[h!]
\centering
\renewcommand{\arraystretch}{1.3}
\begin{tabular}{c|c|c|}
\cline{2-3}
 \, & $\kappa_\eps \in (-b_\Sigma; -1 )$  & $\kappa_\eps \in (-1;-1/b_\Sigma)$  \\
\cline{2-3}
 \, & $s^\out(r,\theta)= s^-(r,\theta) = r^{-i\eta}\Phi(\theta)$ & $s^\out(r,\theta)=s^+(r,\theta) = r^{+i\eta}\Phi(\theta)$ \\
 \hline
 \multicolumn{1}{|c|}{{$0 <\phi < \pi$}} 
& $\Phi$ given by (\ref{PhiSkewSym}) skew-symmetric & $\Phi$ given by (\ref{PhiSym}) symmetric\\
 \hline
\multicolumn{1}{|c|}{{$\pi <\phi < 2\pi$ }}
&  $\Phi$ given by (\ref{PhiSym}) symmetric & $\Phi$ given by (\ref{PhiSkewSym}) skew-symmetric \\
 \hline 
\end{tabular}
\caption{Features of the outgoing singularity with respect to the configuration. 
\label{tableSummary}}
\end{table}
\begin{remark}\label{rmk:Vphi}
Note that when $\kappa_\eps \in (-1;-1/b_\Sigma)$, the wave $s^\out e^{-i\omega t}$ (see (\ref{WavesBidule})) has a positive phase velocity and seems to come from the corner. However, $s^\out$ propagates energy towards the corner. We stress that we select the physical solution according to the group velocity and not according to the phase velocity.
\end{remark}

Let us briefly present another approach, which has been used in \cite{BoChCl13} (see also \cite{Nguy15} in a slightly different context), to define the ``physical'' singularity. We emphasize that it leads to select the same solution.

\subsection{Selecting the outgoing solution II: limiting absorption principle}\label{subsectionOutgoing}

We recall that the original Drude's model \eqref{lossyDrudeModel} includes a small parameter $\gamma $ which takes into account classical dissipation Joule effects. We point out that we choose $\gamma\ge 0$ so that, with the convention of a harmonic term equal to $e^{-i\om t}$, energy is indeed lost by the structure (the alternative convention $e^{i\om t}$ leads to take $\gamma\le 0$ in order to model dissipation). We denote by $\eps^\gamma_{\mrm{m}}$ the permittivity obtained with this model and we define $\eps^\gamma$ the function such that $\eps^\gamma=\eps_{\mrm{d}}$ in $\Om_{\mrm{d}}$, $\eps^\gamma=\eps^\gamma_{\mrm{m}}$ in $\Om_{\mrm{m}}$. The smallness of $\gamma $ compared to the considered range of frequencies has led us to neglect it in the analysis and this is the reason of the difficulties we have encountered. Indeed, when $\gamma >0 $, $\kappa_{\eps^{\gamma}}:=\eps^\gamma_{\mrm{m}}/\eps_{\mrm{d}}\not \in \eR $ and one can easily check that the functions $f^\pm$ defined in Proposition \ref{prop:fonction_zeros} with $\kappa_{\eps}$ replaced by $\kappa_{\eps^{\gamma}}$ do not vanish on $ (0; + \infty)$. In other words, purely oscillating singularities do not occur with dissipation. More mathematically, when $\eps$ is changed to $\eps^{\gamma}$, the new sesquilinear form associated with Problem (\ref{eq:FV_Sommerfeld_exacte}) becomes coercive in $H^1(D_R)$. Therefore, the dissipative problem always admits a unique solution, denoted $u^{\gamma}$, in this space. The function $u^{\gamma}$ decomposes near the corner as 
\[
u^{\gamma} = a^{\gamma}+b^{\gamma} s^{\gamma}+\tilde{u}^{\gamma},
\]
where $a^{\gamma}$, $b^{\gamma}$ are constants, $\tilde{u}^{\gamma}$ is a smooth function and $s^{\gamma}(r,\theta)=r^{\lambda^\gamma}\Phi^{\gamma}(\theta)$. Here, $\lambda^\gamma$ is the singular exponent of smallest positive real part of $\Lambda^{\gamma}$, the set of values of $\lambda$ such that (\ref{pb_phi}), with $\eps$ replaced by $\eps^{\gamma}$, has a non zero solution. The following result confirms the relevance of choosing the outgoing singularity $s^{\out}$ according to Table \ref{tableSummary}. 

\begin{proposition}\label{propositionLimitingAbsorption}
Assume that $\kappa_{\eps}\in (-b_{\Sigma};-1/b_{\Sigma})\setminus\{-1\}$. Then $s^{\out}=\lim_{\gamma\to 0}s^{\gamma}$, where $s^\out $ is defined according to Table \ref{tableSummary}.
\end{proposition} 
Let us sketch the proof. Denote $\lambda^{\gamma}_{\pm}\in \Lambda^{\gamma}$ the singular exponent which tends to $\pm i\eta$ as $\gamma$ goes to zero. Introduce $\hat{\lambda}_{\pm}$ the first order term appearing in the Taylor expansion $\lambda^{\gamma}_{\pm}=\pm i\eta+\gamma \hat{\lambda}_{\pm}+\dots$. Using the implicit functions theorem, one can prove that $\hat{\lambda}_{\pm}$ are real valued, $\hat{\lambda}_{+}=-\hat{\lambda}_{-}$ and $\hat{\lambda}_{+}\int_{-\pi}^{\pi}\eps^{-1} |\Phi|^2\,d\theta <0$. Assume that $\kappa_{\mrm{\eps}}\in(-b_{\Sigma};-1)$.  Then, according to (\ref{relationSpectral}), we have $\int_{-\pi}^{\pi}\eps^{-1} |\Phi|^2\,d\theta >0$. Since by definition $\Re e\,\lambda^\gamma>0$, we deduce that $\lambda^\gamma$ coincides with $\lambda^\gamma_{-}$ and therefore, tends to $-i\eta$ as $\gamma\to0$. This implies $\lim_{\gamma\to 0}s^{\gamma}=s^{-}$. But Table \ref{tableSummary} ensures that $s^{-}=s^{\out}$ when $\kappa_{\mrm{\eps}}\in(-b_{\Sigma};-1)$. The case $\kappa_{\mrm{\eps}}\in(-1;-1/b_{\Sigma})$ can be handled in a similar way.\\
\newline   
In Figure \ref{img:spectre_compa_pml_dissip} (middle), we represent the set $\Lambda^{\gamma}$ for a small value of $\gamma>0$. One observes that the numerical results are in accordance with Proposition  \ref{propositionLimitingAbsorption}. 

\subsection{A well-posed formulation of the scattering problem for a contrast inside the critical interval}\label{ssec:Kondratiev}
At this point, we have provided the ingredients to obtain a well-posed formulation for the scattering Problem (\ref{eq:startingpb}) with a contrast lying in the critical interval. When $\ke \in I_c \setminus \lbrace -1 \rbrace $, we look for solutions $u$ in the sense of distributions of $\R^2$ (denoted $\mathscr{D}'(\R^2)$) which admit the expansion 
\begin{equation}\label{expansionSolution}
u = b s^\out + \tilde{u} \quad \mbox{in } \R^2, \quad \mbox{with } b \in \mathbb{C},\ \tilde{u} \in H^1_{\mrm{loc}}(\R^2),
\end{equation}
where the outgoing singularity $s^\out$ is defined according to Table \ref{tableSummary}. In particular, a solution $u$ satisfies
\[
a(u,w)=l(w),\qquad\forall w\in\mathscr{C}^{\infty}(\overline{D_R}),
\]
where $\mathscr{C}^{\infty}(\overline{D_R})=\{\varphi|_{D_R},\,\varphi\in\mathscr{C}^{\infty}_0(\R^2)\}$ and where $a(\cdot,\cdot)$, $l(\cdot)$ are  defined in (\ref{eq:FV_Sommerfeld_exacte}).
\begin{remark}
Note that $\eps^{-1}\nabla s^{\mrm{out}}$ is an element of $L^1_{\mrm{loc}}(\R^2)^2\subset\mathscr{D}'(\R^2)^2$. Therefore, $\div(\eps\inv \grad s^{\mrm{out}})$ is well-defined in the sense of distributions of $\R^2$. Moreover, one can check that $\div(\eps\inv \grad s^{\mrm{out}})=0$ in $\mathscr{D}'(D_r)$ for $r$ small enough, so that $\div(\eps\inv \grad u)=\div(\eps\inv \grad\tilde{u})$ in $\mathscr{D}'(D_r)$.
\end{remark}
Imposing the specific behaviour (\ref{expansionSolution}) for the solution is like imposing a radiation condition at the corner. As nicely written in \cite{Berr09} for another problem sharing analogous properties, this boils down to allow a ``leak'' at $\bfc$. Now, we prove the well-posedness of the problem in this setting. We start with a uniqueness result whose proof relies again on energy considerations.

\begin{lemma} \label{cor:uniqueness_Kondra}
Problem (\ref{eq:startingpb}) has at most one solution admitting decomposition \eqref{expansionSolution}.
\end{lemma}
\begin{proof}
Consider some $u$ admitting decomposition \eqref{expansionSolution} and satisfying Problem (\ref{eq:startingpb}) with $u^{\mrm{inc}}=0$. Multiplying the volume equation of (\ref{eq:startingpb}) by $\overline{u}$, integrating by parts in $D_R \setminus \overline{D_\rho}$ and taking the imaginary part, we get the energy balance
\begin{equation}\label{J-J}
\Im m\left(\int_{\partial D_R} \eps_\di^{-1}\partial_r u\, \overline{u} \, d\sigma  \right)=\Im m\left(\int_{\partial D_\rho} \eps^{-1}\partial_r u\, \overline{u} \, d\sigma  \right).
\end{equation}
Denote $J_\ext$ (resp. $J$) the left-hand side (resp. right-hand side) of (\ref{J-J}). We have selected $s^{\out}$ so that 
\[
J=-\eta\,\vert b \vert^2  \left| \dsp \int_{-\pi}^{\pi} \eps\inv \Phi^2 \, d\theta\right| \le0
\]
(see the discussion after (\ref{relationSpectral})). Therefore, from (\ref{J-J}) we deduce that $J_\ext\le0$. This is true also with $R$ replaced by $\xi\ge R$. Then using identity (\ref{DvpSommerfeld}) and working as in the proof of Lemma \ref{lem:uniqueness} with Rellich's lemma, we obtain $u=0$ in $\R^2$.
\end{proof}

The proof of existence of a solution requires more involved arguments based on the Kondratiev theory \cite{Kond67} and is beyond the scope of the present article. We refer the reader to \cite{BoChCl13,BoCh13} where a detailed explanation of the technique is presented in a simple geometry. Finally we can state the
\begin{proposition}\label{propInsideInterval}
Let $\om>0$ be a given frequency. Assume that the contrast $\kappa_{\eps}=\e{\text{m}}/\e{\text{d}}$ verifies $\kappa_{\eps}\in (-b_{\Sigma};-1/b_{\Sigma})\setminus \{-1\}$, where $b_\Sigma$ has been defined in \eqref{defParamIntervalle}. Then Problem (\ref{eq:startingpb}) has a unique solution $u$ admitting decomposition \eqref{expansionSolution}. Moreover there exists $C>0 $ independent of the data $g^{\mrm{inc}} $ such that
\[\vert b \vert + \Vert \tilde{u} \Vert_{H^1(D_R)} \leq C \Vert g^{\mrm{inc}} \Vert_{L^2(\partial D_R).}\]
\end{proposition}
\begin{remark}\label{rmk:volumic_source}
The results of Proposition \ref{propInsideInterval} can be extended to 
consider Problem (\ref{eq:startingpb}) with a volume equation replaced by $\div (\eps \inv \grad u) +k^2_0 \mu u = f$, $f$ being a given source term. Well-posedness is ensured if $f$ has a compact support and if $f$ is such that $r^{1-\nu} f \in L^2(\R^2) $ for some $\nu >0 $. In particular $f \in L^2(\R^2)$ with a compact support is allowed.
\end{remark}
\noindent Let us make some comments to conclude this section: \\[5pt]
$\bullet$ In the recent paper \cite{Li14}, the author suggests that the good way to formulate the scattering problem for a contrast inside the critical interval is to choose $u$ in the vicinity of the corner such that $J$ defined in \eqref{def_J} vanishes. This is attractive because in this case, the metallic scatterer neither absorbs nor produces energy (like in (\ref{eq:DR})). To get such a solution, one must keep both incoming and outgoing singularities, with the balancing condition $\vert b^+ \vert = \vert b^- \vert $. In other words, $u$ must decompose as
\begin{equation}\label{new_decomp}
u = b (s^+ + e^{it} s^-) + \tilde{u},
\end{equation}
where $b \in \mathbb{C} $, $t\in [0,2\pi)$ and where $\tilde{u} $ is a smooth function. In the present article, we did not use this criterion for the following reasons. First, there is still an undetermined parameter to set, namely the phase $t$. Second, the limiting absorption principle, which can be rigorously proven working as in \cite[Theorem 4.3]{BoChCl13}, is satisfied in the setting \eqref{expansionSolution} but not in the setting \eqref{new_decomp}. Therefore, it is our opinion that the decomposition \eqref{expansionSolution} is more relevant from a physical point of view than \eqref{new_decomp}. \\
$\bullet$ When the contrast and the interface are such that for the $N$ vertices $\bfc_1,\dots,\bfc_N$, there exist oscillating singularities $s^{\pm}_n(r_n,\theta_n):=r_n^{\pm i\eta_n}\Phi_n(\theta_n)$ at $\bfc_n$, ($\eta_n >0 $) the analysis is exactly the same. Here, $(r_n,\theta_n)$ denote the polar coordinates associated with $\bfc_n$. In this case, we can prove that Problem (\ref{eq:startingpb}) has a unique solution $u$ which admits the expansion
\begin{equation}\label{representationSolutionMulti}
u=\sum_{n=1}^N  b_n\, s^{\out}_n  +\tilde{u}\ \mbox{ in }\R^2,\qquad \mbox{with } b_n \in\Cplx\mbox{ and } \tilde{u}\in H^1_{\mrm{loc}}(\R^2).
\end{equation}
In that case, working as in \eqref{J-J}, denoting $\partial D^n_\rho:=\{\bfx\in\R^2\,|\,|r_n|=\rho\}$, we obtain the energy balance
\begin{equation}\label{J-JN}
J_\ext = \sum \limits_{n=1}^N J_n,\quad\mbox{ with }J_n:=\Im m\left(\int_{\partial D^n_\rho} \eps^{-1}\partial_{r_n} u\, \overline{u} \, d\sigma  \right)=-\eta_n\,\vert b_n \vert^2  \left| \dsp \int_{-\pi}^{\pi} \eps\inv \Phi_n^2 \, d\theta_n\right|.
\end{equation}
Using \eqref{J-JN}, one can quantify the energy trapped by each corner (see \S \ref{ssec:energy} for numerical illustrations).

\section{Approximation of the solution for a contrast inside the critical interval}\label{sec:PMLs}

We have obtained a well-posed formulation for Problem (\ref{eq:startingpb}) with a contrast inside the critical interval. It leads to look for solutions $u$ which decompose as $u= \sum_{n}b_n\,s_n^{\out} +\tilde{u}$ (see (\ref{representationSolutionMulti})). Now, a natural question is: how to approximate $u$? Let us try to use a classical finite element method. We consider a setting where the inclusion is a triangle made of silver embedded in vacuum. The angles of the triangle (see Figure \ref{img:FemDic}) are equal to  $\phi_1 = \pi/6$ (top corner) and $\phi_2 = \phi_3 = 5\pi/12$ (bottom corners). For such geometry, according to (\ref{defParamIntervalle}) and \eqref{eq:Ic}, we have $b_\Sigma =  (2\pi-\pi/6)/(\pi/6) = 11$ so that the critical interval is given by 
\[
I_c=[ -11; -1/11].
\]
For silver, the plasma frequency is $\omega_p = 13.3 \text{ PHz}$ \cite{Palik85}. From the dissipationless Drude's model \eqref{eq:drude}, we deduce that
\[
\kappa_\eps \in I_c \quad \Longleftrightarrow \quad \omega \in \left[ \frac{\omega_p}{\sqrt{1 + b_\Sigma}} ; \frac{\omega_p}{\sqrt{1 + 1/b_\Sigma}}\right] = [3.839\text{ PHz} ; 12.733\text{ PHz}].
\]
For our experiment, we set $\omega  = 9\text{ PHz}$ (corresponding to $\e\me (\omega) = -1.1838$), $\e\di = 1$, $\mu_\me = \mu_\di = 1$. Therefore, we have $\kappa_{\eps}=\e\me/\e\di=-1.1838 \in [-11;-1]$. For the other parameters, we take
\[ 
u^{\mrm{inc}}(\bfx) = e^{i \textbf{k} \cdot \bfx},\quad \textbf{k}  = k  \left( \cos \alpha^{\mrm{inc}} \overrightarrow{e_x} + \sin \alpha^{\mrm{inc}} \overrightarrow{e_y}\right),\quad k= k_0 = \omega/c, \qquad \mbox{ and } \quad \alpha^{\mrm{inc}} = -\pi/12.
\]

\begin{figure}[!h]
\centering
\includegraphics[width=0.28\columnwidth]{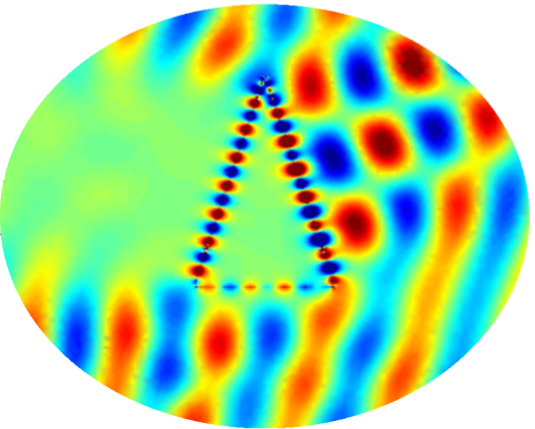}\qquad
\includegraphics[width=0.28\columnwidth]{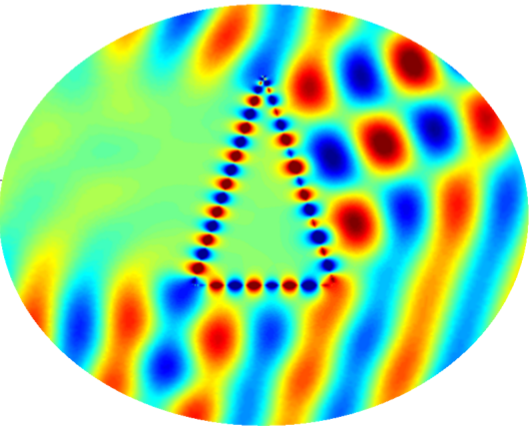}\qquad 
\includegraphics[width=0.28\columnwidth]{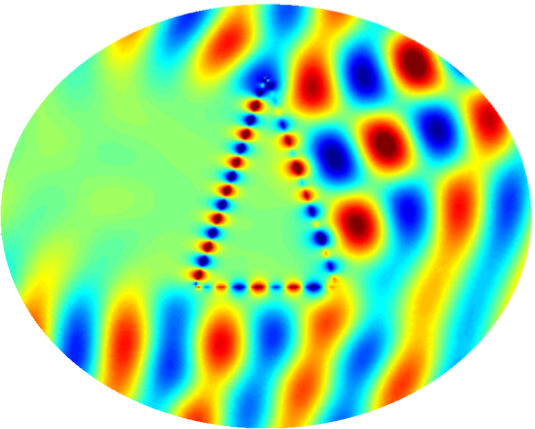}
\caption{Solution obtained using a standard P2 finite element method for different meshes: from left to right, 13273 nodes, 56031 nodes and 100501 nodes. The frequency is set to $\omega = 9$ PHz ($\kappa_{\eps}=-1.1838 $). Note that the computed field is not stable at the interface when we refine the mesh. \label{img:FemDic}}
\end{figure}

\noindent In Figure \ref{img:FemDic}, we represent the approximated total field obtained using a standard P2 finite element method for three different meshes. The incident plane wave produces both a usual scattered field outside the inclusion and a typical plasmonic wave at the interface between the two materials. When we refine the mesh, the scattered field outside the inclusion is approximately stable. However, the plasmonic wave seems very sensitive to the mesh (see in particular at the bottom and right edges of the inclusion).  The numerical solution does not converge when the mesh size tends to zero, the classical finite element method fails to approximate the field which is not in $H^1$ locally around the corners. More precisely, this is due to the fact that it is impossible to capture the oscillations of the singularities $s^{\out}_n$ (see Figures \ref{dessin propagative singu coupe 1D}, \ref{img:black-hole_waves}) with a mesh of given size. Spurious reflections are always produced. Hence, we have to develop another method.

\subsection{Analogy with a waveguide problem}\label{paragraphAnalogyWaveguide}

\begin{figure}[!ht]
\centering
\def\svgwidth{0.85\columnwidth}
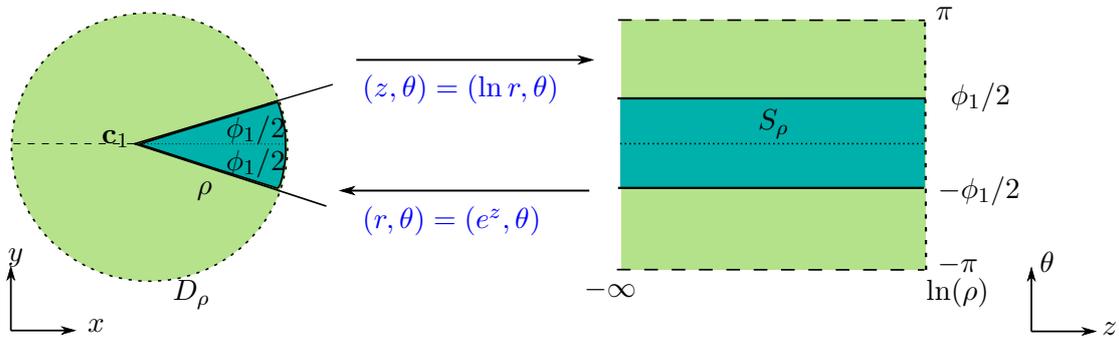
\caption{Change of variables at the corner. The disk $D_\rho $ is transformed into a semi-infinite strip $S_\rho$ (the waveguide) with periodic conditions in $\theta$. As $r\rightarrow 0 $, $z$ tends to  $-\infty $.\label{img:zoom}}
\end{figure}

In order to capture the oscillations of $s^{\out}_n$, a natural idea consists in unfolding a neighbourhood of each corner using a change of variables. To explain the idea, assume that there is only one corner $\bfc$. Define $z=\ln r$ (as it is classical for the Mellin transform \cite{Kond67}), $\breve{u}(z,\theta)=u(e^{z},\theta)$. In a neighbourhood of $\bfc$, the function $\eps $ depends only on $\theta$ so we make no difference between $\eps(r,\theta)$ and $\eps(z,\theta)$. With this notation, as illustrated by Figure \ref{img:zoom}, Equation (\ref{eq:pb_polar_coord}) in $D_{\rho}$ is changed into the equation
\[
\big( {\eps}^{-1}\partial^2_z+\partial_{\theta} {\eps}^{-1} \partial_{\theta}\big)\breve{u} + k^2_0 e^{2 z} \mu \breve{u} =0
\]
in the semi-infinite strip (the waveguide) $S_\rho:=(-\infty;\ln\rho)\times(-\pi;\pi)$. Note that $z=\ln r\rightarrow- \infty $ when $r\rightarrow 0$. As a consequence, the corner is sent to $-\infty$ in the waveguide. With this change of variables, the function $\breve{u} $ is $2\pi$-periodic in $\theta $: $\breve{u}(\cdot, -\pi) =  \breve{u}(\cdot, \pi) $ and $\partial_\theta \breve{u}(\cdot, -\pi) =  \partial_\theta \breve{u}(\cdot, \pi) $. On the other hand, the term $k^2_0 e^{2 z} \mu$ is exponentially decaying as $z\rightarrow - \infty $. As a consequence, the behaviour of $\breve{u} $ at $-\infty$ is determined by the functions $\breve{s}$ which satisfy
\[
\big( {\eps}^{-1}\partial^2_z+\partial_{\theta} {\eps}^{-1} \partial_{\theta}\big)\breve{s} = \div (\eps \grad \breve{s}) =0 \quad \mbox{in } S_\rho
\]
and which are $2\pi$-periodic for the $\theta$ variable. Since $r=e^z$, the singularities $s(r,\theta)=r^{\lambda}\Phi(\theta)$, solutions of \eqref{pbVariablesSeparees} in $D_{\rho}$, are turned into $\breve{s}(z,\theta)=e^{\lambda z}\Phi(\theta)$ in $S_\rho$. These functions are commonly called the modes of the waveguide $S_\rho$. In \S \ref{ssec:singular_expo}, we said that for $\lambda \in \Lambda $ such that $\Re e\,\lambda >0 $ the singularity $s=r^{\lambda}\Phi$ belongs to $H^1(D_\rho)$. In this case, the associated mode $\breve{s}=e^{\lambda z}\Phi$ is evanescent in the waveguide $S_\rho $. While, for $\lambda= \pm i \eta$, $ \eta>0$, the oscillating singularities $s^\pm$ do not belong to $H^1(D_\rho)$. The corresponding modes $\breve{s}^\pm :=e^{\pm i \eta z}\Phi$ are propagative in $S_\rho $. According to Lemma \ref{lem:exist_ilambda}, we know that propagative modes exist only for contrasts inside the critical interval. In the presence of propagative modes, it is well-known that a radiation condition at infinity in the waveguide has to be enforced to obtain a well-posed problem. But we have already done this work for the corner problem. Define  $\breve{s}^\out$ such that $\breve{s}^\out(z,\theta)=s^\out(e^z,\theta)$. Then in the waveguide $S_\rho $, we look for solutions $\breve{u}$ which decompose as $\breve{u} = b \breve{s}^\out + \breve{u}_{\mathrm{ev}}$, $b\in\Cplx$, where $\breve{u}_{\mathrm{ev}}$ is the sum of a constant term and evanescent modes at $-\infty$.\\
\newline
For numerical purposes, we will use the analogy with the waveguide writing a formulation of the scattering problem in a domain split in two parts, namely the perforated domain $D_R/\overline{D_\rho}$ and the semi-infinite strip $S_\rho$. The main difficulty lies in the fact that the new geometry is unbounded and that the solution we want to approximate does not decay at infinity in $S_\rho$. However, many efficient techniques have already been developed to consider waveguide problems in presence of propagative modes. A class of methods consists in bounding artificially the waveguide to compute an approximation of the solution on a bounded domain. For this kind of approaches, it is well known that the waveguide has to be bounded in a clever way to avoid spurious reflections on the artificial boundary. One technique to achieve this end is to use a Perfectly Matched Layer (PML) \cite{Ber98,BeBoLe04}. In the following, we apply this method to our problem. First, we set up the PML. Then, we explain how to truncate the PML to derive a formulation set in a bounded domain which can be discretized numerically.

\subsection{An approximation of the scattering problem at the continuous level}\label{ChoicePMLparam}

Imposing a PML in the semi-infinite strip $S_\rho $ boils down to compute an analytic continuation of $\breve{u}$. In practice it leads to make the complex stretching $z\mapsto z/\alpha$, $\alpha\in\Cplx\setminus\{0\}$. With this stretching one finds that $\breve{u}_\alpha(z,\theta) := \breve{u}(z/\alpha, \theta) $ satisfies
\[ \big(\alpha^{2} {\eps}^{-1}\partial^2_{z}+\partial_{\theta}{\eps}^{-1} \partial_{\theta}\big)\breve{u}_\alpha+e^{2z/\alpha }k_0^2 {\mu} \breve{u}_\alpha=0 \quad \mbox{in }S_\rho .
\]
Let us explain how to choose the parameter $\alpha\in\Cplx\setminus\{0\}$. Without loss of generality, we can impose $|\alpha|=1$ so that $\alpha=e^{i\vartheta}$ for some $\vartheta\in(-\pi;\pi]$. In order for the function $z\mapsto e^{2z/\alpha}$ to be exponentially decaying at $-\infty$, we impose $\Re e\,\alpha>0$ which amounts to take $ \vartheta\in(-\pi/2;\pi/2)$. On the other hand, observe that the modes of the problem  $\big(\alpha^{2}\breve{\eps}^{-1}\partial^2_{z}+\partial_{\theta}\breve{\eps}^{-1} \partial_{\theta}\big)\breve{w}=0$ with periodic boundary conditions for the $\theta$ variable are the functions $(z,\theta)\mapsto e^{\lambda z/\alpha}\Phi(\theta)$, where $(\lambda,\Phi)$ corresponds to an eigenpair of Problem (\ref{pb_phi}). And if $\breve{u}$ decomposes on the modes $e^{\lambda z}\Phi(\theta)$, $\breve{u}_\alpha$, the analytic continuation of $\breve{u}$, decomposes on the modes $e^{\lambda z/\alpha}\Phi(\theta)$. Therefore, in order $\breve{u}_\alpha$ to be exponentially decaying at $-\infty$, we have to choose $\alpha$ such that there holds $\Re e\,(\lambda/ \alpha)>0$ for all $\lambda\in\Lambda^{\out}\setminus\{0\}$. Here, $\Lambda^{\out}$ refers to the set of singular exponents appearing in the modal decomposition of $\breve{u}$:
\[
\Lambda^\out :=\lbrace 0, \lambda^\out \rbrace\cup\tilde{\Lambda}^{\out}\qquad\mbox{ with }\quad  \tilde{\Lambda}^{\out}=\{\lambda \in \Lambda\,| \, \Re e\,\lambda >0 \rbrace,
\]
where $\lambda^\out$ denotes the singular exponent of $s^\out$ defined according to Table \ref{tableSummary}. This means that we choose $\vartheta $ such that $\pi/2 + \arg(\lambda) > \vartheta > -\pi/2 + \arg(\lambda) $ for all $\lambda \in \Lambda^\out \setminus \lbrace 0 \rbrace$, where $\arg:\Cplx\setminus\{0\}\to(-\pi;\pi]$ denotes the complex argument. Let us clarify this. \\
\newline
$\star$ When $\kappa_{\eps}\in(-b_{\Sigma};-1)$, according to Table \ref{tableSummary}, we have $-i\eta\in \Lambda^{\out}$ and $i\eta\notin \Lambda^{\out}$. In this case, one takes $\alpha$ such that $\Re e\,(-i\eta/\alpha)>0$, that is $ \Im m\,(\alpha)<0$. Then, we choose $\alpha=e^{i\vartheta}$ with $\vartheta\in(\vartheta_-;0)$, where $\vartheta_- :=-\pi/2+\max \limits_{\lambda \in \tilde{\Lambda}^{\out}} \arg(\lambda)$. \\
\newline
$\star$ When $\kappa_{\eps}\in(-1;-1/b_{\Sigma})$, we have $i\eta\in \Lambda^{\out}$ and $-i\eta\notin \Lambda^{\out}$. Working as above, we find that a good choice for $\alpha$ is $\alpha=e^{i\vartheta}$ with $\vartheta\in(0;\vartheta_+)$, where $ \vartheta_+:=\pi/2+ \min  \limits_{\lambda \in \tilde{\Lambda}^{\out}} \arg(\lambda)$.
\begin{figure}[!h]
\centering
\includegraphics[width=0.3\columnwidth]{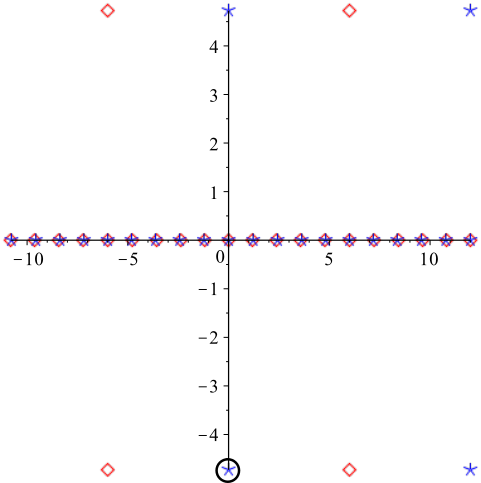}\quad\includegraphics[width=0.3\columnwidth]{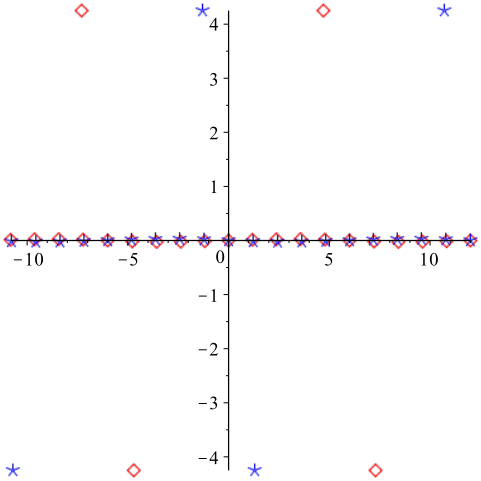}\quad\includegraphics[width=0.33\columnwidth]{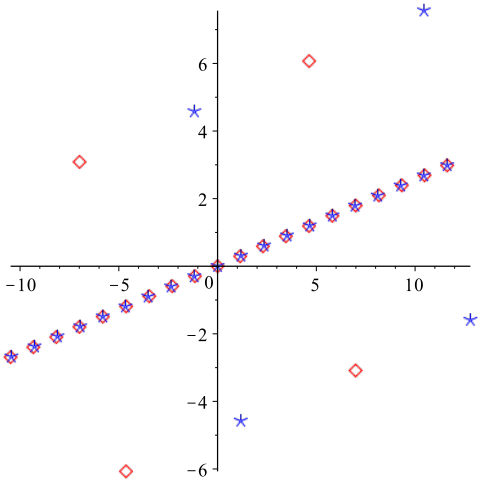}
\caption{Left: set of singular exponents $\Lambda =\Ls \cup \Lss$ (diamonds for $\Ls $ and asteriks for $\Lss $) for $\omega =9$ PHz ($\ke = -1.1838 \in I_c $) and a corner aperture $\phi = \pi/6 $. The singular exponent associated to $s^\out $ is circled in black. In accordance with the result of Proposition \ref{propositionLimitingAbsorption}, the outgoing singularity $s^\out $ in this setting is $s^- $ (and it is skew-symmetric). Middle: set $\Lambda^{\gamma}$ for small $\gamma$ ($\gamma$ is the dissipation). Right: set of singular exponents with the PML.  \label{img:spectre_compa_pml_dissip}}
\end{figure}

\noindent In Figure \ref{img:spectre_compa_pml_dissip}, we display the sets of singular exponents for the problem with small dissipation and for the problem with a PML. The important point is that these two regularization processes move the singular exponent $\lambda^\out$ in the half plane $\{\lambda\in\Cplx\,|\,\Re e\,\lambda>0\}$. This suggests that our PML parameter $\alpha$ has been correctly set. For the problem with a PML, all the modes except the constant one are evanescent. Therefore, for numerical purposes, we can truncate the waveguide $S_\rho$. Define the domain $S^{L}_\rho=(\ln \rho-L;\ln\rho)\times(-\pi;\pi)$ with $L>0$. On the boundary $\{\ln \rho -L\}\times (-\pi;\pi)$, we impose the Neumann condition $\partial_z \breve{u}=0$ to allow constant behaviour at $-\infty$ (the PML has no influence on the constant mode). We emphasize that the Dirichlet boundary condition would produce spurious reflections. We also introduce the parameter $L_0\in(0;L)$ such that the term $ k_0^2 \breve{\mu} e^{2z} \breve{u}$ becomes neglectable for all $z \leq \ln \rho -L_0 $ and we define the function ${\alpha} $ such that
\begin{equation}\label{defParaPML}
 {\alpha}(z) := \left\{\begin{array}{ll} 1 &  \mbox{ for } z \in (\ln \rho - L_0; \ln \rho),\\
\alpha &  \mbox{ for } z \in (\ln \rho -L; \ln\rho - L_0).
\end{array}\right.
\end{equation}
Working in the waveguide enables to dilate the radial coordinate near the corner and $L_0$ defines the beginning of the PML. The previous analysis leads us to consider the problem
\begin{equation}\label{eq:PbPML1}
\begin{array}{|rcll}
\multicolumn{4}{|l}{\text{Find } (u,\breve{u}) \in H^1(D_R\setminus \overline{D_{\rho}})\times H^1_{\mrm{per}}(S^L_{\rho})\text{ such that: }}\\[6pt]
\div\left({\eps}^{-1}\nabla u\right) + k_0^2 \mu u & = & 0 & \text{ in } D_R\setminus \overline{D_{\rho}} ,\\[6pt]
 \partial_{r} u -\mathcal{S} u & =  &   g^{\mrm{inc}} & \text{ on } \partial D_R, \\[6pt]
\big({\eps}^{-1}\partial_z {\alpha} \partial_z+{\alpha} \inv \partial_{\theta}{\eps}^{-1} \partial_{\theta}\big)\breve{u}+ {\alpha}\inv  e^{2z/\alpha }k_0^2 \mu \breve{u}& = & 0 &  \text{ in } S^{L}_{\rho} ,\\[6pt]
\partial_z \breve{u}(\ln \rho -L,\cdot) & = & 0,\\[6pt]
\multicolumn{4}{|l}{u(\rho,\cdot)=\breve{u}(\ln\rho,\cdot),\ \rho\partial_{r} u(\rho,\cdot)=\partial_{z} \breve{u}(\ln\rho,\cdot),}
\end{array}
\end{equation}
where $H^1_{\mrm{per}}(S^L_{\rho})$ denotes the Sobolev space of functions of $H^1(S^L_{\rho}) $ which are $2\pi$-periodic for the $\theta$ variable. We recall that $g^{\mrm{inc}} = \partial_{r}u^{\mrm{inc}} - \mathcal{S} u^{\mrm{inc}}$. Note that the above problem  is set in a split domain. The last two equations of \eqref{eq:PbPML1} ensure the matching between $u$ and $\breve{u}$ through $\partial D_\rho $. The variational formulation associated with (\ref{eq:PbPML1}) is given by
\begin{equation}\label{eq:PbPML2}
\begin{array}{|lcl}
\text{Find } (u,\breve{u}) \in X\text{ such that: } \\[6pt]
b_1(u,v)+b_2(\breve{u},\breve{v}) = l(v), \qquad \forall (v,\breve{v})\in X,
\end{array}
\end{equation}
with $X :=\{ (v,\breve{v}) \in H^1(D_R\setminus \overline{D_{\rho}})\times H^1_{\mrm{per}}(S^L_\rho) \,|\,v(\rho,\cdot)=\breve{v}(\ln\rho,\cdot)\}, \qquad l(v) =  \dsp \int_{\partial D_R} \frac{g^{\mrm{inc}}}{\e\di} \overline{v} \, d \sigma$  
\begin{equation}\label{defFormPML}
\begin{array}{l}
b_1(u,v) =  \dsp\int_{D_R\setminus\overline{D_\rho}} \eps^{-1}\nabla u \cdot \overline{\nabla v}\,d\bfx - k_0^2\dsp\int_{D_R\setminus\overline{D_\rho}}  \mu \, u \overline{v}\,d\bfx- \e\di\inv \sum  \limits_{n= -\infty}^{+\infty}   k \,\Frac{ H^{(1)'}_n (kR)}{H^{(1)}_n (kR)} u_n \overline{v_n} \\[14pt]
b_2(\breve{u},\breve{v}) =  \dsp\int_{S^L_\rho} {\alpha}\, {\eps}^{-1}\partial_z\breve{u}\,\partial_z\overline{\breve{v}}
+\alpha\inv \,{\eps}^{-1}\partial_{\theta}\breve{u}\,\partial_{\theta}\overline{\breve{v}}\,d\bfx - k_0^2\dsp\int_{S^L_\rho} {\alpha}\inv \, e^{2z/\alpha }\, \mu\,\breve{u}\,\overline{\breve{v}}\,d\bfx .
\end{array}
\end{equation}

\subsection{Numerical approximation}
\begin{figure}[!h]
\centering
\includegraphics[width=0.47\columnwidth]{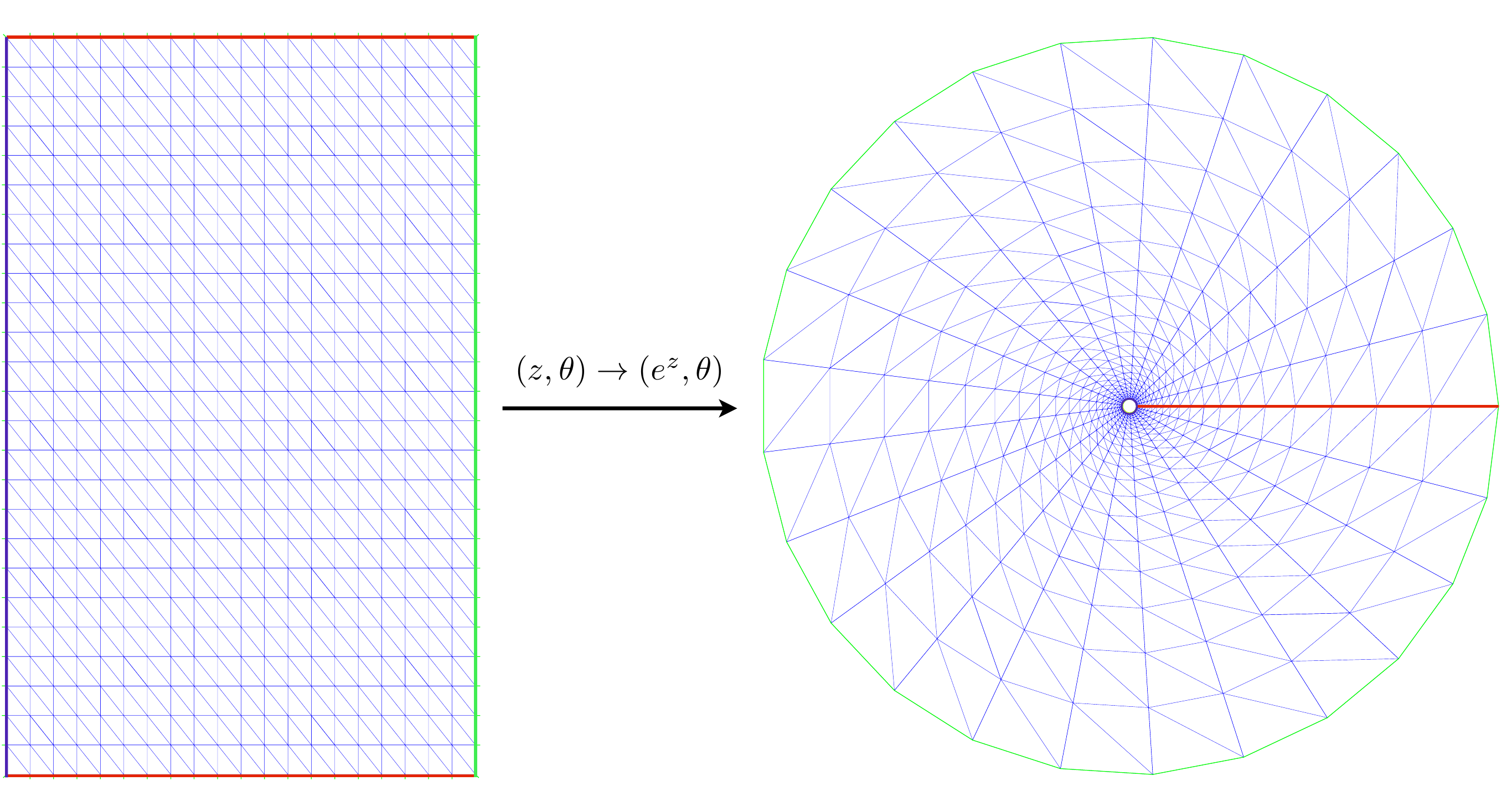} 
\caption{Left: mesh of $S^L_\rho $. Right: obtained mesh with the change of variables $(z,\theta)\mapsto (e^z,\theta)$.\label{figs_mesh_Euler}}
\end{figure}

Now we turn to the discretization of Problem \eqref{eq:PbPML2}. We use a $\mrm{P}2$ finite element method set on the domain which is the union of the perforated disk $D_{R} \setminus \overline{D_\rho}$ and the rectangle $S^L_\rho$. We introduce $(\mathcal{T}_h,\breve{\mathcal{T}}_h)_h$ a shape regular family of triangulations of $(D_{R,h} \setminus \overline{D_{\rho,h}},S^{L}_{\rho})$ where $D_{j,h}$ is a polygonal approximation of $D_{j}$, $j = R$, $\rho$. Here $h$ refers to the mesh size. We assume that the meshes are such that all the triangles of $(\mathcal{T}_h,\breve{\mathcal{T}}_h)_h$ are located either in the dielectric or in the metal. Note that considering a structured mesh for $S^L_\rho$ boils down to work with a mesh having a logarithmic structure near the corner (see Figure \ref{figs_mesh_Euler}). Due to the behaviour of the black-hole singularities, this is of course very interesting. To get (almost) conforming approximations of $X$, we impose that the nodes of $\mathcal{T}_h$ located on $\partial D_{\rho,h}$ coincide with the ones of $\breve{\mathcal{T}}_h$ situated on $\{\ln\rho\}\times[-\pi;\pi]$ (see Figure \ref{figs_mesh}). Then we define the family of finite element spaces
\[
\begin{array}{l}
X_h := \{(v, \breve{v})\in H^1(D_{R,\,h}\setminus \overline{D_{\rho,h}})\times H^1_{\mrm{per}}(S^L_{\rho})\mbox{ such that } v(\rho,\cdot)=\breve{v}(\ln\rho,\cdot),\\[2pt]
\phantom{X_h := \{}  \mbox{ and } (v|_{\tau},\breve{v}|_{\tau'})\in \mathbb{P}_2(\tau)\times\mathbb{P}_2(\tau')\mbox{ for all }(\tau,\tau')\in\mathcal{T}_h\times\breve{\mathcal{T}}_h\}.
\end{array}
\]
Practically, we compute the solution of the problem
\begin{equation}\label{eq:discrete_pbPML}
\begin{array}{|l}
\text{Find } (u_h,\breve{u}_h) \in X_h \text{ such that: } \\[6pt]
b_{1,h}(u_h,v_h) + b_{2}(\breve{u}_h,\breve{v}_h)= l_h(v_h) \quad \forall (v_h, \breve{v_h}) \in X_h.
\end{array}
\end{equation}
In (\ref{eq:discrete_pbPML}), the forms $b_{1,h}(\cdot,\cdot)$, $l_h(\cdot)$ are defined as $b_{1}(\cdot,\cdot)$, $l(\cdot)$ (see (\ref{defFormPML})) with $D_R\setminus\overline{D_\rho}$ replaced by $D_{R,h} \setminus \overline{D_{\rho,h}}$ and the approximate transparent boundary condition  $\partial_r u -(ik-(2R)^{-1})u = \partial_r u^{\mrm{inc}} -(ik-(2R)^{-1})u^{\mrm{inc}}$ instead of the Dirichlet-to-Neumann map $\mathcal{S}$. 
  
\begin{figure}[!h]
\centering
\includegraphics[width=0.35\columnwidth]{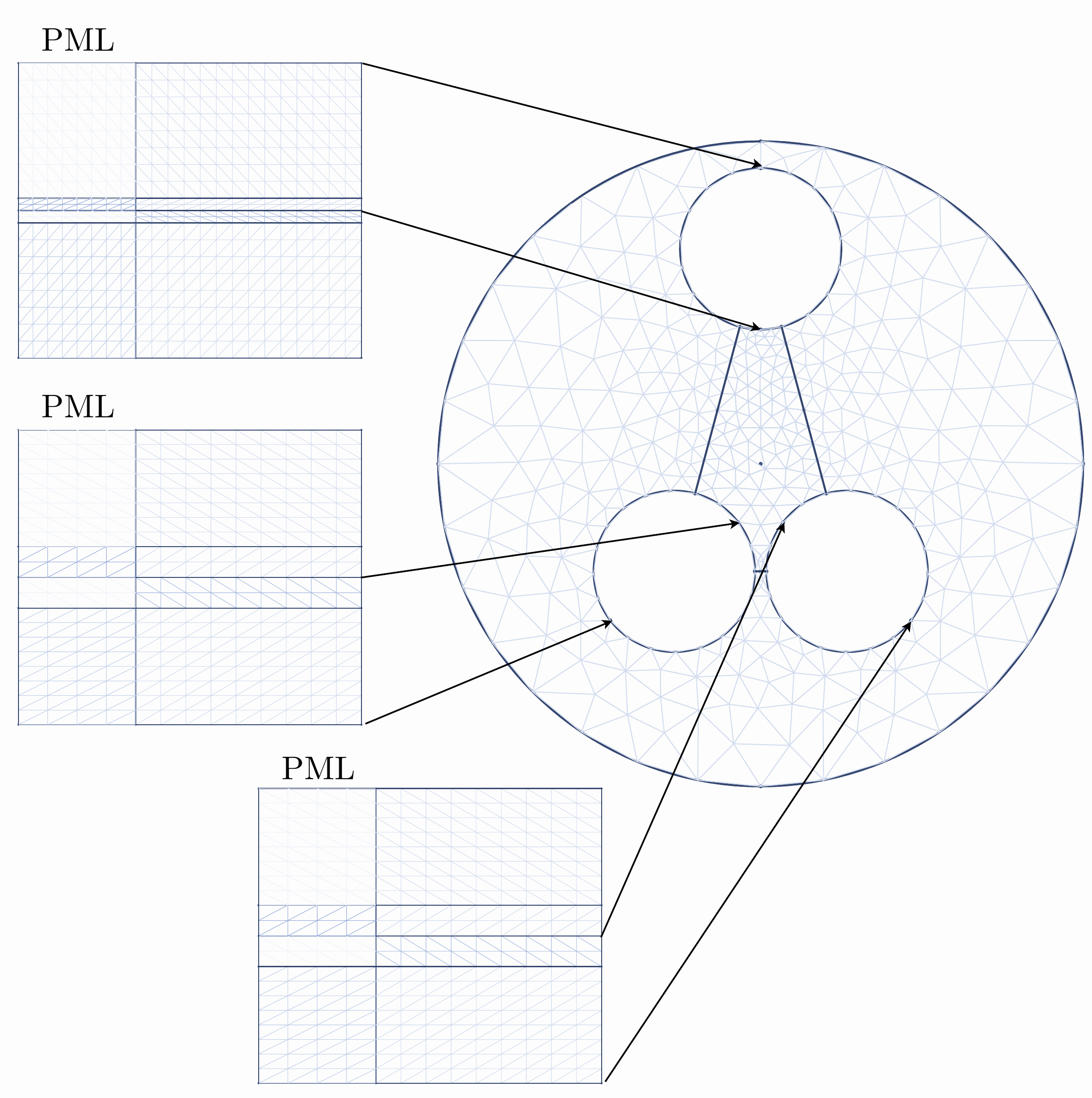} 
\caption{Example of (very coarse) mesh. The arrows show the matching nodes between the different meshes. \label{figs_mesh}}
\end{figure}

\begin{remark}
In the case where the metallic inclusion has $N$ corners, one  discretizes the problem 
\begin{equation}\label{eq:PbPML3}
\begin{array}{|rcll}
\multicolumn{4}{|l}{\mrm{Find }\ (u,\breve{u}_1, \dots,\breve{u}_N ) \in H^1(D_R \setminus \cup {}_{i=1}^{N} \overline{D_{\rho_i}})\times H^1_{\mrm{per}}(S_1^{L_1}) \times \dots \times  H^1_{\mrm{per}}(S_N^{L_N})\ \mrm{ such\ that}:}  \\[6pt]
\div\left({\eps}^{-1}\nabla u\right) + k_0^2 \mu u  & = & 0 & \mrm{ in }\ D_R\setminus \cup {}_{i=1}^{N} \overline{D_{\rho_i}}\\[6pt]
 \partial_{r} u -S u  & = &  g^{\mrm{inc}} & \mrm{ on }\ \partial D_R \\[6pt]
\big(   {\eps}^{-1}\partial_z{\alpha}_i \partial_z+{ {\alpha}_i}\inv \partial_{\theta} {\eps}^{-1} \partial_{\theta}\big)\breve{u}_i+ {{\alpha}_i}\inv  e^{2z {{\alpha}_i}\inv}k_0^2 \mu \breve{u}_i & = & 0  & \mrm{ in }\ S_i^{L_i}, \quad i=1,\dots,N \\[6pt]
\partial_z \breve{u}(\ln \rho-L_i,\cdot) & = & 0, & i=1,\dots,N \\[6pt]
\multicolumn{4}{|l}{u(\rho_i,\cdot)=\breve{u}_i(\ln\rho_i,\cdot),\ \rho_i \partial_{r} u(\rho_i,\cdot)=\partial_{z} \breve{u}_i(\ln\rho_i,\cdot),  \quad i=1,\dots,N.} 
\end{array}
\end{equation}
In (\ref{eq:PbPML3}), for $i=1,\dots,N $, $D_{\rho_i} $ is a small disk around ${\bfc}_i$ of radius $\rho_i$ while $S_i^{L_i}:=(\ln \rho-L_i;\ln\rho_i)\times(-\pi;\pi)$ with $L_i>0$ (see Figure \ref{figs_mesh}). For each PML, we use a parameter $\alpha_i$ as in (\ref{defParaPML}). 
\end{remark}

\subsection{Numerical experiments}
\subsubsection{Numerical results with the PML approach}

We present some results obtained with the analogous of Formulation \eqref{eq:discrete_pbPML} for Problem (\ref{eq:PbPML3}). We consider the same setting as in the beginning of \S\ref{sec:PMLs}. We choose PML coefficients such that $\alpha_1 = e^{-i 2\pi/25} $ (top corner) and $\alpha_2 = \alpha_3 = e^{-i 2\pi/33}$ (bottom corners). We consider an incident field of incidence $\alpha^{\mrm{inc}} = -\pi/12 $. In Figure \ref{figs_DIC_pml}, we observe that the numerical solution seems to converge when we refine the mesh, contrary to what happens without the PMLs (see Figure \ref{img:FemDic}). In Figure \ref{img:FemDicPML}, we display the field inside the PMLs. We note that it is correctly attenuated and that at the end of the PMLs, the solution seems constant. According to Table \ref{tableSummary}, for $\ke =  -1.1838$, the outgoing singularity at each corner is skew-symmetric with respect to the corner's bisectors. This is indeed the case. For this particular incidence, at the bottom right corner the solution is locally symmetric (with respect to the corner's bisector) and there is no excitation of the outgoing singularity.

\begin{figure}[!h]
\centering
\includegraphics[width=0.99\columnwidth]{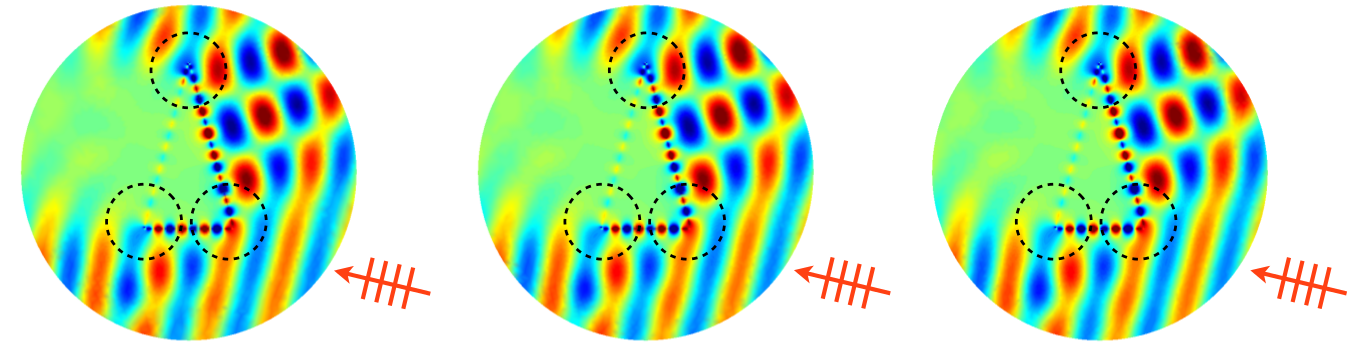}

\caption{Numerical solution in $D_R$  for $\omega = 9$ PHz ($\ke=-1.1838$). We reconstitute the field inside the three disks using the transform $(z,\theta)\mapsto (e^z,\theta)$. As we refine the mesh (from left to right), the solution does not change much.\label{figs_DIC_pml}}
\end{figure}

\begin{figure}[!h]
\centering
\includegraphics[width=0.6\columnwidth]{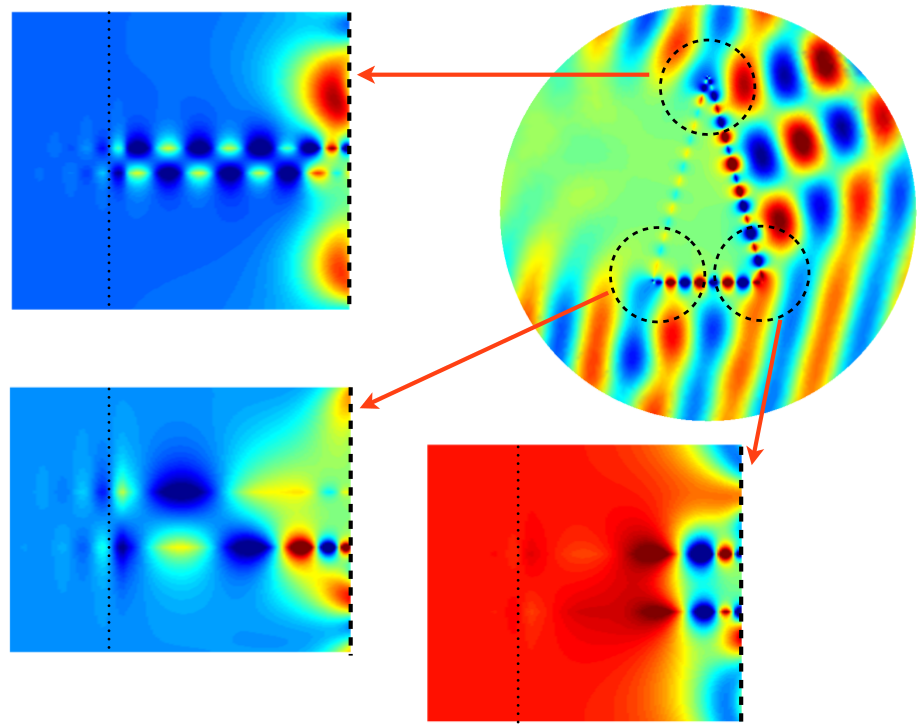}
\caption{Numerical solution in $D_R$ with PMLs for $\omega = 9$ PHz ($\ke=-1.1838$). The bold dashed lines correspond to the interfaces where matching is made, the small dotted lines represent the boundary of the PMLs. Note that for the chosen incidence $\alpha^{\mrm{inc}}=-\pi/12$ ($330$\textdegree), there is no trapped energy at the right corner (see also Figure \ref{img:Energy_flux_triangle}). \label{img:FemDicPML}}
\end{figure}

\subsubsection{Energy conservation}\label{ssec:energy}
In this section, we wish to give numerical illustrations of the energy balance $J_\ext = \sum_{n} J_n$ obtained in (\ref{J-JN}). Using (\ref{eq:Sommerfeld_exacte}), (\ref{J-J}) and (\ref{J-JN}), we find 
\begin{equation}\label{eq: flux_cons_N}
J_\ext = \Im \text{m} \left( \int_{\partial D_R} \eps_\di^{-1}(\mathcal{S}(u-u^{\mrm{inc}}) + \partial_r u^{\mrm{inc}}) \overline{u}\,d\sigma\right)\quad\mbox{ and }\quad J_n=-\eta_n\,\vert b_n \vert^2  \left| \dsp \int_{-\pi}^{\pi} \eps\inv \Phi_n^2 \, d\theta_n\right|,
\end{equation}
where $\mathcal{S}$ is the Dirichlet-to-Neumann map on $\partial D_R$. To assess the term $J_n$, the energy trapped at the corner $\bfc_n$, we have to derive formulas to compute the coefficient $b_n$. There are two different approaches to do that. Let us present the two methods when the interface has only one corner $\bfc $. \\
\newline 
\noindent$\star$ By definition, $u$ admits the expansion $u = b s^\out + \tilde{u}$ where $b \in \mathbb{C}$ and $\tilde{u} \in H^1_{\mrm{loc}}(\R^2)$. This yields 
\[ 
\dsp \int_{\partial D_\rho} \eps\inv  u \,\overline{s^\out}\, d \sigma = b \,  \dsp \int_{\partial D_\rho} \eps\inv  s^\out\,\overline{s^\out}\, d \sigma + \dsp \int_{\partial D_\rho} \eps\inv  \tilde{u}\,\overline{s^\out}\, d \sigma.
\]
Proceeding as in \S \ref{proof_lemma_energy_DIC}, one finds that $|  \int_{\partial D_\rho} \eps\inv  \tilde{u} \,\overline{s^\out}\, d \sigma | \leq C \rho^{\beta}$ for some $\beta >0$. We deduce
\begin{equation}\label{b1}
 b =  \dsp\frac{\dsp\rho^{-\lambda^\out}\int_{-\pi}^{\pi} \eps\inv  u(\rho,\cdot) \Phi\, d\theta}{\dsp \int_{-\pi}^{\pi} \eps\inv \Phi^2\, d\theta} + O(\rho^{\beta}).
\end{equation}
In (\ref{b1}), $\lambda^\out$ denotes the singular exponent of $s^\out$ defined in Table \ref{tableSummary}. Note that according to Lemma \ref{ResultatCalcul}, we know that the denominator of the above equation does not vanish.\\
\newline
\noindent$\star$ Let us present another approach to assess the coefficient $b$. We follow a classical idea to compute stress intensity factors. First, introduce $\mathfrak{s}$ the solution to the following problem:
\begin{equation}\label{solu_b}
\begin{array}{|rcll}
\multicolumn{4}{|l}{\mbox{Find } \mathfrak{s}=\zeta s^\inc + c s^\out + \tilde{\mathfrak{s}},\mbox{ with }c\in\Cplx\mbox{ and }\tilde{\mathfrak{s}}\in H^1(D_R), \mbox{ such that:}}\\[3pt]
\div (\eps\inv \grad \mathfrak{s}) +k^2_0 \mu\,  \mathfrak{s} &=&  0  &\mbox{in } D_R,\\[3pt]
\partial_r \mathfrak{s} - \mathcal{S} \mathfrak{s} & = &0  & \mbox{on } \partial D_R.
\end{array}
\end{equation}  
In (\ref{solu_b}), $\zeta$ is a given cut-off function such that $\zeta=1$ in a neighbourhood of $\bfc$ and such that $ \div (\eps\inv \grad s^\inc)=0$ on the support of $\zeta$. Looking for a solution $\mathfrak{s}$ of \eqref{solu_b} is equivalent to look for $w=\mathfrak{s}- \zeta s^\inc $ solution of
\begin{equation}\label{solu_ws}
\begin{array}{|rcll}
 \div (\eps\inv \grad w) +k^2_0 \mu\,  w &=& f_1^\inc := -(\div (\eps\inv \grad(\zeta s^\inc)) +k^2_0 \mu\,  (\zeta s^\inc))& \mbox{in } D_R,\\[3pt]
 \partial_r w - \mathcal{S} w &=& f_2^\inc :=-(\partial_r(\zeta s^\inc )- \mathcal{S}(\zeta s^\inc )) & \mbox{on } \partial D_R.
\end{array}
\end{equation}  
Using Remark \ref{rmk:volumic_source}, one can easily prove that Problem \eqref{solu_ws} has a unique solution. Solving \eqref{solu_ws} consists in solving \eqref{eq:Sommerfeld_exacte} with a source that, instead of coming from the exterior domain, comes from the corner. Now, if $u = b s^\out + \tilde{u}$ with $b \in \mathbb{C}$, $\tilde{u} \in H^1_{\mrm{loc}}(\R^2)$ is a solution of \eqref{eq:startingpb}, by Green's formula we get
\begin{equation}\label{calculusCoef1}
\dsp \int_{ \partial D_R} \eps \inv \left( \partial_r u \,\mathfrak{s} - \partial_r \mathfrak{s}\,u \right)\, d\sigma -\int_{ \partial D_\rho} \eps \inv \left( \partial_r u \,\mathfrak{s} - \partial_r \mathfrak{s}\, u \right)\, d\sigma = 0.
\end{equation}
Proceeding again as in \S\ref{proof_lemma_energy_DIC}, one finds 
\begin{equation}\label{calculusCoef2}
\int_{ \partial D_\rho} \eps \inv \left( \partial_r u \,\mathfrak{s} - \partial_r \mathfrak{s} \,u \right)\, d\sigma  =2  b  \, \lambda^\out \dsp \int_{-\pi}^{\pi}\eps\inv\Phi^2\, d\theta + O(\rho^\beta)\qquad\mbox{ for some }\beta>0.
\end{equation}
On the other hand, we have $\int_{ \partial D_R} \eps \inv \left( \partial_r u\,\mathfrak{s} - \partial_r \mathfrak{s}\,u \right)\,d\sigma  = \int_{\partial D_R} \e\di^{-1}g^{\mrm{inc}} \mathfrak{s} \, d\sigma$
(use (\ref{eq:DtN}) and the properties of the Hankel functions to obtain this). Plug the latter identity and (\ref{calculusCoef2}) in (\ref{calculusCoef1}). Then take the limit as $\rho \rightarrow0$. We obtain 
\begin{equation}\label{eqnDefCoefSingu}
b = \dsp \frac{ \dsp \int_{\partial D_R} \e\di^{-1}g^{\mrm{inc}} \mathfrak{s} \, d\sigma}{2 \lambda^\out \dsp \int_{-\pi}^{\pi}\eps\inv \Phi^2 \, d\theta} .
\end{equation}
\noindent The advantage of the second approach is twofold. First, it does not require to compute the solution $u$. The function $\mathfrak{s}$ can be approximated once for all, independently from the source term. Moreover, numerically, the second method is more accurate than the first one.\\ 
\newline
Let us turn to numerical simulations. We compute the terms $J_\ext$, $J_n$ (see (\ref{eq: flux_cons_N})) for $\alpha^{\mrm{inc}}\in [0;2\pi)$. We perform two series of experiments: one with $\omega = 9 $ PHz ($\ke=-1.1838$), another one with $\omega = 11 $ PHz ($\kappa_\eps = -0.4619$). For the latter case, according to Table \ref{tableSummary} and \S\ref{ChoicePMLparam}, we have to change the PML coefficients. We take $\alpha_1 = e^{i 2\pi/25} $ (top corner) and $\alpha_2 = \alpha_3 = e^{i 2\pi/33}$ (bottom corners). All the other parameters are set as previously (see the beginning of \S\ref{sec:PMLs}). In Figure \ref{img:Energy_flux_triangle}, we work with $\omega = 9 $ PHz ($\ke=-1.1838$). We observe that the energy balance \eqref{eq: flux_cons_N} seems to be satisfied. There is a small mismatch between $J_\ext$ and $\sum_n J_n$. Probably, this is because we use the first method described above to assess the coefficients $b_n$ appearing in the definition of $J_n$. Remark that due to the symmetry of the geometry, the results are symmetric (the left and right corners play a similar role). One notice that for $\alpha^{\mrm{inc}}  = \pm \pi/2$ ($90$\textdegree~and $270 $\textdegree), there is no trapped energy at the top corner. This was expected. Indeed, for this setting, according to Table \ref{tableSummary}, we know that the black-hole singularities are skew-symmetric with respect to the bisector of $\bfc_n$. But for these two particular incidences, $u^{\mrm{inc}}$ is symmetric with respect to the top corner's bisector. As a consequence, there is no excitation of the outgoing singularity. The same phenomenon occurs for the other corners when $\alpha^{\mrm{inc}}$ corresponds to the direction of the bisector of $\bfc_n$. \\
In Figure \ref{img:Energy_sym_triangle}, we work with $\omega = 11 $ PHz ($\ke=-0.4619$). When $\ke>-1 $, according to Table \ref{tableSummary}, the black-hole waves are symmetric. This explains why this time, we observe that the energy trapped at the corner $\bfc_n$ ($n=1,2,3$) is maximum when $\alpha^{\mrm{inc}}$ coincides with the direction of the bisector of $\bfc_n$. 

\begin{figure}[!h]
\centering
\includegraphics[width=0.38\columnwidth]{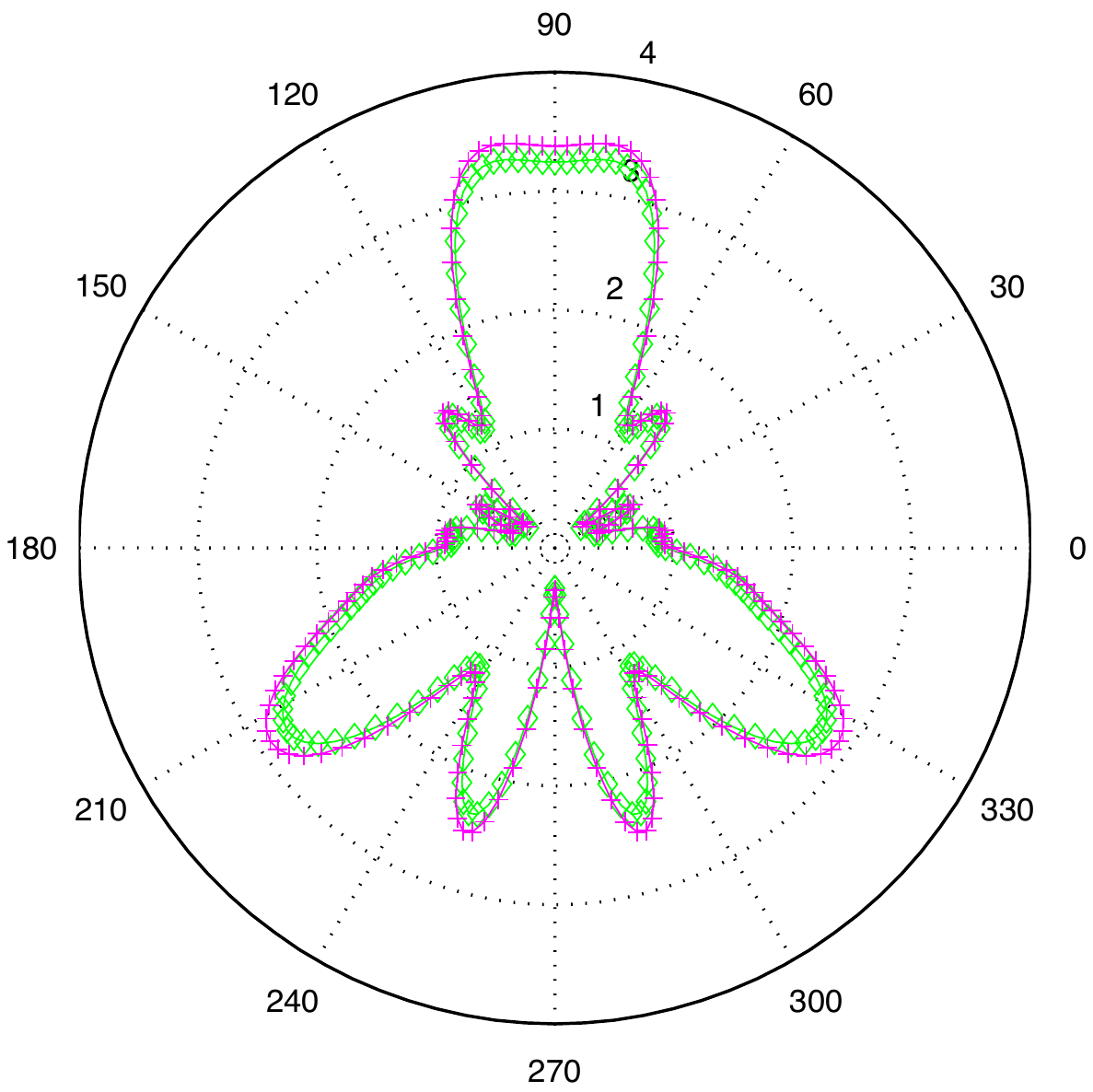}\qquad
\includegraphics[width=0.38\columnwidth]{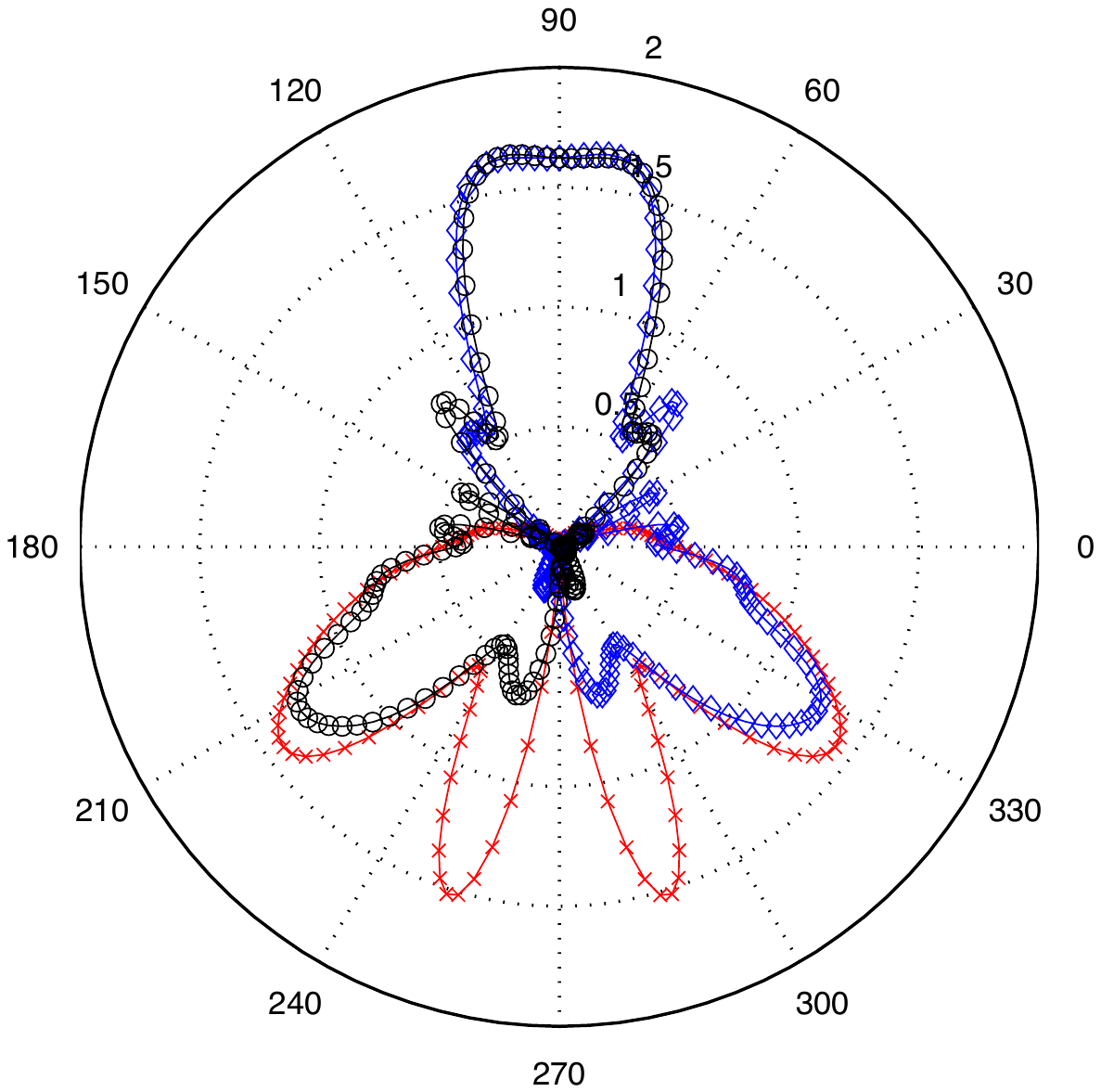}
\caption{Left: sum of the energy fluxes at the corners (diamond) and energy flux through $\partial D_R$ (cross) with respect to $\alpha^{\mrm{inc}}\in[0;2\pi)$ in polar coordinates. Right: energy flux at the top corner (cross), left corner (diamond), right corner (circle) in polar coordinates. The frequency is set to $\omega = 9$ PHz ($\ke=-1.1838$). \label{img:Energy_flux_triangle}}
\end{figure}
\begin{figure}[!h]
\centering
\includegraphics[width=0.38\columnwidth]{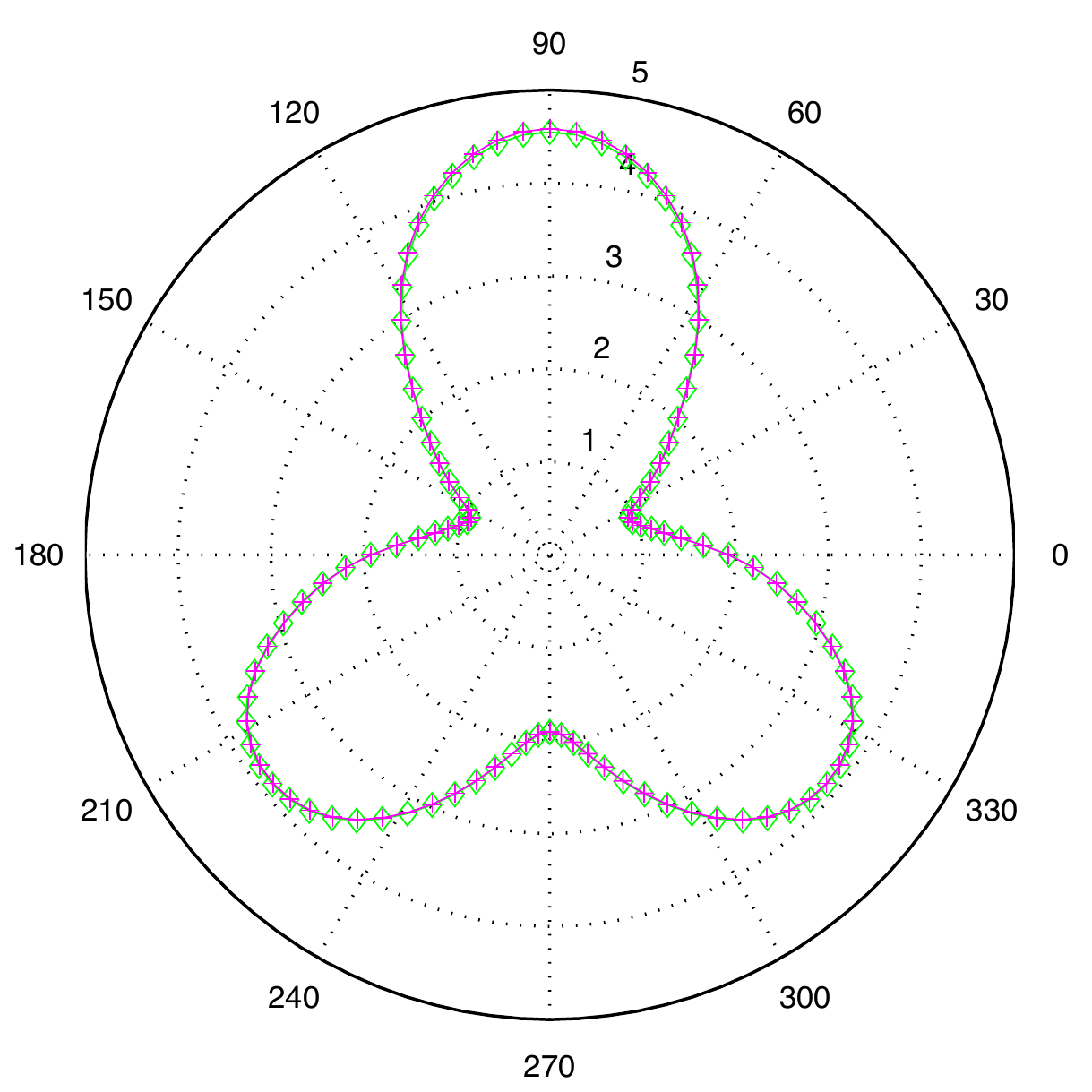}\qquad
 \includegraphics[width=0.38\columnwidth]{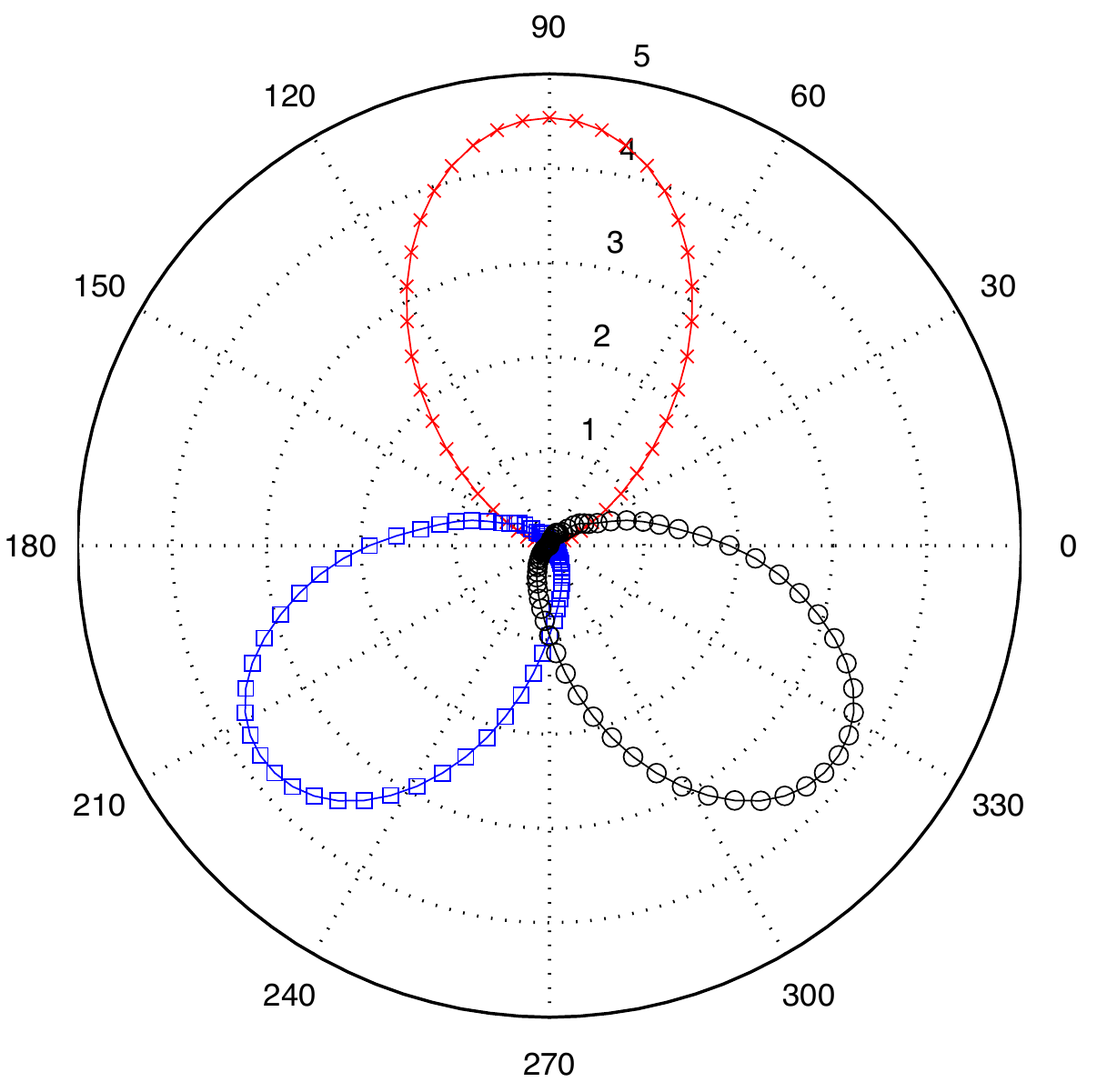}
\caption{Left: sum of the energy fluxes at the corners (diamond) and energy flux through $\partial D_R$ (cross) with respect to $\alpha^{\mrm{inc}}\in[0;2\pi)$ in polar coordinates. Right: energy flux at the top corner (cross), left corner (square), right corner (circle) in polar coordinates. The frequency is set to $\omega = 11 $ PHz ($\kappa_\eps = -0.4619$). \label{img:Energy_sym_triangle}}
\end{figure}

\section{Discussion and prospects}\label{sectionDiscussion}

Let us conclude this paper by making some comments regarding this new numerical method:\\[5pt]
$\bullet$ We point out that the method with PMLs at the corners is also interesting when the metal is slightly absorbing. In this case, the scattering problem is well-posed in the usual $H^1$ framework like when the contrast $\kappa_{\eps}$ lies outside the critical interval. This is due to the fact that there is no oscillating singularities. However, when the dissipation is small, the field can be very singular (according to Proposition \ref{propositionLimitingAbsorption}, we know that $\lim_{\gamma\to 0}s^{\gamma}=s^{\out}$ where $\gamma$ corresponds to the dissipation). As a consequence, it is necessary to use a very refined mesh to obtain a good approximation of the solution. Adding some PMLs at the corners allows to attenuate the singularities without producing spurious reflections. In Figure \ref{figs_dissip}, we use the lossy Drude's model (see (\ref{lossyDrudeModel})) at the frequency $\omega = 6 $ PHz for silver. It yields $\eps^{\gamma}_\me(\omega) = -3.9193 +0.0926 i $. We set the other parameters as previously (with PML coefficients as in the case $\omega = 9 $ PHz). We observe that when the mesh is refined, the numerical solution is much more stable with PMLs than without them. Note that instabilities with respect to dissipation have been already pointed out in \cite{HePhMi11}. 

\begin{figure}[!h]
\centering
\includegraphics[width=\columnwidth]{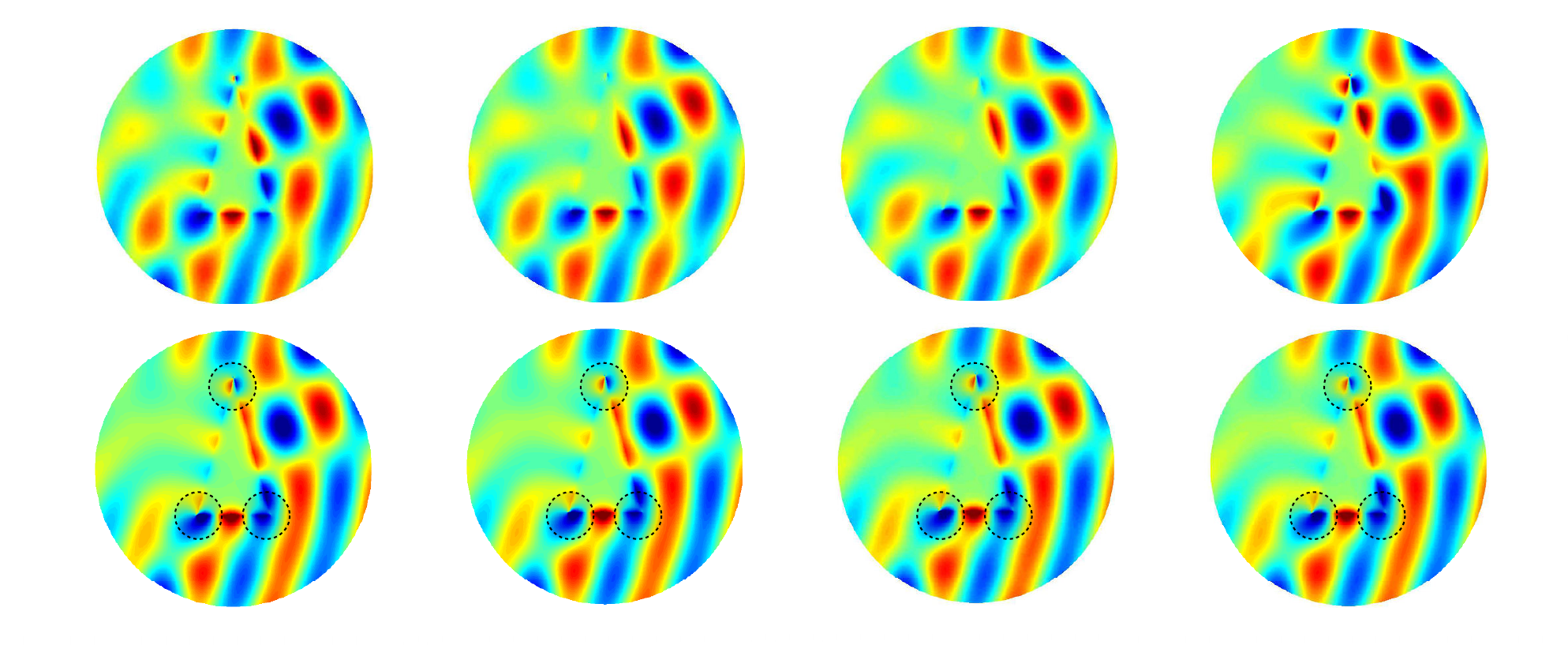}
\caption{Up (resp. bottom): solutions without (resp. with) PMLs. From left to right, around 19 400, 28 000, 42 300 and 76 000 degrees of freedom.\label{figs_dissip}}
\end{figure}

\noindent$\bullet$ In practice, the technique with PMLs turns out to be very efficient to obtain a good approximation of the plasmonic waves propagating at the interface. Concerning the justification of the method, several questions remain open. First, one needs to control the error made when truncating the PMLs. This is not straightforward because of the change of sign of $\eps$. However, using Kondratiev spaces and working for example like in \cite{Kal11}, one can reasonably hope to establish such a result. The problem of the justification of the convergence of the finite element methods seems more intricate. Without PMLs, the existing proofs (with sign-changing $\eps$) require assumptions on meshes (see \cite{ChCi13,BoCaCi14}) and the question of knowing whether or not these assumptions are necessary is not solved. Here, due to the complex scaling $z\mapsto z/\alpha $ of the PMLs, it is not even clear that the continuous problem admits a unique solution.\\ 

\noindent$\bullet$ Some efforts will be devoted to the singularity extraction techniques for problems with a sign-changing coefficient, namely the accurate computation of the coefficient $b$ in (\ref{eqnDefCoefSingu}), in the spirit of \cite{AsCS00,HaLo02,YoAS02}. 

\section{Annex}

\subsection{Proof of decomposition (\ref{eq:T_coerc})}\label{Annex-Tcoercivity} 
Let $\mathcal{B}$ denote the operator associated with the sesquilinear form $b(\cdot,\cdot)$ defined by (\ref{bilinear_b}). Assume that the contrast $\kappa_{\eps}=\e{\text{m}}/\e{\text{d}}$ satisfies $\kappa_{\eps}\notin I_c=[-b_{\Sigma};-1/b_{\Sigma}]$ where $b_\Sigma$ has been set in \eqref{defParamIntervalle}. We prove here the existence of a bounded operator $\tT:H^1(D_R)\to H^1(D_R)$ such that 
\[
\mathcal{B}\circ\tT = \mathcal{I}+\mathcal{K}\qquad\mbox{(decomposition (\ref{eq:T_coerc}))}
\]
where $\mathcal{I}:H^1(D_R)\to H^1(D_R)$ is an isomorphism and where $\mathcal{K}:H^1(D_R)\to H^1(D_R)$ is a compact operator.\\
\newline
Let us introduce $R_0$ such that $R_0<R$ and $\overline{\Omega_m}\subset D_{R_0}$. For a given $w\in H^1_0(D_{R_0})$, we consider the following variational problem:
\begin{equation}\label{eq:FV_dansDR0}
\begin{array}{|lcl}
\text{Find } u_w \in H^1_0({R_0}) \text{ such that: } \\[6pt]
\dsp\int_{D_{R_0}} \eps^{-1}\nabla u_w \cdot \overline{\nabla v}\,d\bfx + i\dsp \int_{D_{R_0}}u_w \overline{v} \,d\bfx=\e{\text{d}}^{-1}(w,v)_{H^1(D_{R_0})}, \qquad \forall v \in H^1_0({R_0}).
\end{array} 
\end{equation}
The uniqueness of the solution is obvious (taking $w=0$ and the imaginary part of the identity for $v=u_w$). 
Then, using the hypothesis on the contrast $\kappa_{\eps}$, we deduce from Theorem 4.3 and Remark 4.4 of 
\cite{BoChCi12} that Problem (\ref{eq:FV_dansDR0}) is well-posed. As a consequence, the map $\tT_0:H^1_0(D_{R_0})\to H^1_0(D_{R_0})$ such that $\tT_0w=u_w$, where $u_w$ is the solution to (\ref{eq:FV_dansDR0}), is an isomorphism.\\ 
Introduce a smooth cut-off function $\chi$ such that $\chi=1$ in a neighborhood of $\Omega_m$ and $\chi=0$ outside $D_{R_0}$. Note that for all $ u,\,v\in H^1(D_R)$, by definition of $\tT_0$, one has:
\begin{equation}
\int_{D_{R}} \eps^{-1}\nabla(\tT_0(\chi u))\cdot \overline{\nabla(\chi v) }\, d\bfx + i\dsp \int_{D_{R}}\tT_0(\chi u)\,\overline{\chi v}\,d\bfx=\e{\text{d}}^{-1}(\chi u,\chi v)_{H^1(D_{R})}.
\label{eq:T0}
\end{equation}
Now, we define the bounded operator $\tT:H^1(D_R)\to H^1(D_R)$ such that for all $u\in H^1(D_R)$,
\[
\tT u=\chi \tT_0(\chi u) + (1-\chi^2)u.
\]
A direct calculation yields 
\[
\begin{array}{lll}
\dsp b(\chi \tT_0 (\chi u),v)&\dsp = \e{\text{d}}^{-1}(\chi u,\chi v)_{H^1(D_{R})} + c_1(u,v),\\
\dsp b((1-\chi^2)u,v)&\dsp = \int_{D_{R}} \e{\text{d}}^{-1}  (1-\chi^2)\nabla u\cdot\overline{\nabla v}\,d\bfx + c_2(u,v) + \e\di\inv  \sum \limits_{n= -\infty}^{+\infty} \Frac{ n}{R}\, u_n\overline{v_n}
\end{array}
\]
with 
\[
\begin{array}{lll}
\dsp c_1(u,v)&\dsp =\int_{D_{R}} \eps^{-1}\nabla \chi \cdot(\tT_0(\chi u)\,\overline{\nabla v}-\nabla(\tT_0(\chi u))\,\overline{v})\,d\bfx- i\dsp \int_{D_{R}}\tT_0(\chi u)\,\overline{\chi v}\,d\bfx,\\[9pt]
\dsp c_2(u,v)&\dsp =\int_{D_{R}} \e{\text{d}}^{-1}\nabla(1-\chi^2)\cdot(u\overline{\nabla v})\,d\bfx.
\end{array}
\]
Finally, setting
\[
\begin{array}{lll}
\dsp c_0(u,v)&\dsp =\int_{D_{R}} \e{\text{d}}^{-1}
((\chi^2+|\nabla \chi|^2-1)u\,\overline{v}+\chi\nabla\chi\cdot(\nabla u \,\overline{v}+u\overline{\nabla v}))\,d\bfx,\\
\end{array}
\]
we get the following identity:
\[
b(\tT u, v)= \e{\text{d}}^{-1}(u, v)_{H^1(D_{R})}+ c_0(u,v)+c_1(u,v)+c_2(u,v) + \e\di\inv  \sum \limits_{n= -\infty}^{+\infty} \Frac{ n}{R}\, u_n\,\overline{v_n} .
\]
We deduce that $\mathcal{B}\circ\tT = \mathcal{I}+\mathcal{K}$ where $\mathcal{I}$, $\mathcal{K}:H^1(D_R)\to H^1(D_R)$ are the bounded operators such that, for all $u$, $v\in H^1(D_R)$, 
\[
\begin{array}{lcl}
(\mathcal{I}u,v)_{H^1(D_R)} &=& \e{\text{d}}^{-1}(u, v)_{H^1(D_{R})}+ \e\di\inv  \dsp\sum \limits_{n= -\infty}^{+\infty} \Frac{ n}{R}\, u_n\,\overline{v_n}\\[10pt]
(\mathcal{K}u,v)_{H^1(D_R)} &=&c_0(u,v)+c_1(u,v)+c_2(u,v).
\end{array}
\]
Now, using classical techniques (in particular, the Lax-Milgram theorem and the compact embedding of $H^1(D_R)$ in $L^2(D_R)$), one can show that $\mathcal{I}$ is an isomophism while $\mathcal{K}$ is a compact operator. Therefore, the proof of (\ref{eq:T_coerc}) is complete.

\subsection{Proof of Proposition \ref{prop:fonction_zeros}}\label{annex:compute}
We reproduce a calculus which can be found in \cite{DaTe97} or \cite[Proposition 3.2.8]{Chesnel12}.  Classically, one can show that $\Phi$ is an eigenfunction associated with the eigenvalue $\lambda\in\Ls $ (resp. $\lambda\in\Lss $) for (\ref{pb_phi}) if and only if it verifies the transmission problem:
\[
\begin{array}{|rcll}
(\lambda^2+\partial^2_{\theta})\Phi & = &   \multicolumn{2}{l}{ 0\qquad\mbox{ on }  (0;\phi/2)} \\[4pt]
(\lambda^2+\partial^2_{\theta})\Phi & = &   \multicolumn{2}{l}{ 0\qquad\mbox{ on }  (\phi/2;\pi)} \\[4pt]
\Phi({\phi^-/2})&=& \multicolumn{2}{l}{\Phi({\phi^+/2}) , \qquad   \eps_{\mrm{m}}^{-1}\partial_{\theta} \Phi({\phi^- /2}) = \eps_{\mrm{d}}^{-1}\partial_{\theta} \Phi({\phi^+ /2})} \\[4pt]
\partial_{\theta} \Phi(0) &=& \partial_{\theta} \Phi(\pi) = 0 &  (\mbox{resp. }  \Phi(0) =  \Phi(\pi) = 0)\\[4pt]
\Phi(\theta)&=& \Phi(-\theta)  & (\mbox{resp. } \Phi(\theta)= -\Phi(-\theta))  \quad\mbox{ on }  (-\pi;0). \\[4pt]
\end{array}
\]
Looking for solutions under the form
\[
\begin{array}{llcl}
 & \Phi(\theta)=A\,\cos (\lambda\theta)\ \mbox{ on } (0;\phi/2) \quad &\mbox{ and } &\quad  \Phi(\theta)=B\,\cos (\lambda(\theta - \pi))\  \mbox{ on } (\phi/2;\pi),\\[4pt]
\mbox{\text{(\ resp.} } & \Phi(\theta)=C\,\sin (\lambda\theta)\ \mbox{ on } (0;\phi/2) \quad & \mbox{ and } & \quad  \Phi(\theta)=D\,\sin (\lambda(\theta - \pi))\ \mbox{ on } (\phi/2;\pi)\ ),
\end{array}
\]
where $(A,B)\ne(0,0)$, $(C,D)\ne (0,0)$ are some constants, we obtain that $\lambda$ belongs to $\Lambda\setminus\{0\}$ if and only if it satisfies 
\[
\kappa_\eps^{-1} \tan(\lambda\phi/2) = \tan (\lambda(\phi/2-\pi)) \qquad (\mbox{resp. }\ \kappa_\eps \tan(\lambda\phi/2) = \tan (\lambda(\phi/2-\pi))).
\]

\subsection{Proof of Lemma \ref{ResultatCalcul}}\label{ssec:calculus}

\noindent The next lemma is a technical result needed in the selection of the outgoing solution (see \S\ref{sec_outgoing}).
\begin{lemmaAnnexe}\label{ResultatCalcul}
Assume that $\kappa_{\eps}\in (-b_{\Sigma};-1/b_{\Sigma})\setminus \{-1\}$. Let $(\pm i\eta, \Phi)$ be a solution of (\ref{pb_phi}), with $\Phi $ equal to (\ref{PhiSkewSym}) or (\ref{PhiSym}) according to the situation. Then we have
\[
\int_{-\pi}^{\pi}\eps^{-1} \Phi ^2\,d\theta>0\quad\mbox{ if}\ \kappa_\eps \in (-b_\Sigma; -1 )\quad \mbox{ and }\quad \int_{-\pi}^{\pi}\eps^{-1} \Phi ^2\,d\theta<0\quad\mbox{ if}\ \kappa_\eps \in (-1;-1/b_\Sigma).
\]
\end{lemmaAnnexe}
\begin{proof}
Set \quad\quad$\aleph_\mrm{m}:=\dsp\int_{|\theta|<\phi/2} \Phi ^2 d\theta$
\quad\quad\textrm{and}\quad\quad$
\aleph_\mrm{d}:=\dsp\int_{\phi/2 < |\theta|<\pi} \Phi ^2 d\theta$.\\\newline
$\star$ Pick some $\ke\in(-b_{\Sigma};-1)$. We want to show that $\eps_{m}^{-1}\aleph_\mrm{m} + \eps_{d}^{-1}\aleph_\mrm{d}>0$. Since 
$\eps_{m}^{-1}\aleph_\mrm{m} + \eps_{d}^{-1}\aleph_\mrm{d} = \eps_{m}^{-1}( \aleph_\mrm{m} + \kappa_{\eps}\aleph_\mrm{d} )>\eps_{m}^{-1}( \aleph_\mrm{m} - \aleph_\mrm{d})$, it is enough to prove that $\aleph_\mrm{m} - \aleph_\mrm{d}<0$.\\[6pt]
- First, assume that $0<\phi<\pi$.  Explicit calculus using the expression of $\Phi$ given by (\ref{PhiSkewSym}) yields
\[
\aleph_\mrm{m}=\cfrac{\sinh(\eta\phi)-(\eta\phi)}{\eta(\cosh(\eta\phi)-1)}
\quad\mbox{ and }\quad
\aleph_\mrm{d}=\cfrac{\sinh(\eta(2\pi-\phi))-(\eta(2\pi-\phi))}{\eta(\cosh(\eta(2\pi-\phi))-1)} \ .
\]
Define $\dsp h(t):=(\sinh t-t)/(\cosh t-1)$. We have $\eta(\aleph_\mrm{m} - \aleph_\mrm{d})=h(\eta\phi)-h(\eta(2\pi-\phi))$, 
so it is sufficient to show that $h$ is an increasing function on $(0;+\infty)$. 
One computes $h'(t)=(2-2\cosh t + t\sinh t)/(\cosh t-1)^2$. Define $g(t)=2-2\cosh t + t\sinh t$. 
One finds $g'(t)=-\sinh t + t\cosh t$ and $g''(t)=t\sinh t$. One deduces, 
successively, $g'>0$ and $h'>0$. Thus $h$ is indeed an increasing function on $(0;+\infty)$.\\[4pt] 
- When $\pi<\phi<2\pi$, using the expression (\ref{PhiSym}), one finds 
\[
\aleph_\mrm{m}=\cfrac{\sinh(\eta\phi)+(\eta\phi)}{\eta(\cosh(\eta\phi)+1)}
\quad\mbox{ and }\quad
\aleph_\mrm{d}=\cfrac{\sinh(\eta(2\pi-\phi))+(\eta(2\pi-\phi))}{\eta(\cosh(\eta(2\pi-\phi))+1)} \ .
\]
Introduce $\dsp \hat{h}(t):=(\sinh t+t)/(\cosh t+1)$. We have $\eta(\aleph_\mrm{m} - \aleph_\mrm{d})=\hat{h}(\eta\phi)-\hat{h}(\eta(2\pi-\phi))$, 
so it is sufficient to prove that $\hat{h}$ is a decreasing function on $(0;+\infty)$. 
One computes $\hat{h}'(t)=(2+2\cosh t - t\sinh t)/(\cosh t+1)^2$. Define $\hat{g}(t)=2+2\cosh t - t\sinh t$. 
One finds $\hat{g}'(t)=\sinh t - t\cosh t$ and $\hat{g}''(t)=-t\sinh t$. One deduces, 
successively, $\hat{g}'<0$ and $\hat{h}'<0$. Thus $\hat{h}$ is indeed a decreasing function on $(0;+\infty)$.\\[6pt]
$\star$ The same approach, \textit{mutatis mutandis}, shows that $
\dsp\int_{-\pi}^{\pi}\eps^{-1} \Phi ^2\,d\theta<0$ when $\ke\in(-1;-1/b_{\Sigma})$.
\end{proof}

\subsection{Details of the proof of Lemma \ref{cor:uniqueness_Kondra}}\label{proof_lemma_energy_DIC}

Let $u=b\,s^{\out} +\tilde{u}$, with $b \in \mathbb{C},\ \tilde{u} \in H^1_{\mrm{loc}}(\R^2)$, be a solution of \eqref{eq:Sommerfeld_exacte}. Lemma \ref{ResultatCalcul} and \eqref{CalculPlusMoins} yield
\begin{equation*}
\dsp \Im m \Big( \int_{\partial D_\rho} \eps\inv \partial_r u \,\overline{u} \, d \sigma \Big) = -\vert b\vert^2 \eta \dsp  \Big|   \int_{-\pi}^{\pi} \eps \inv \Phi^2 \, d\theta \Big| + \Im m \Big(\int_{\partial D_\rho} \eps\inv ( b \, \partial_r s^\out \,\overline{\tilde{u}} + \overline{b} \, \partial_r \tilde{u} \,\overline{s^\out} + \partial_r \tilde{u} \,\overline{\tilde{u}}) \, d \sigma  \Big).
\end{equation*}
To obtain the result of Lemma \ref{cor:uniqueness_Kondra}, we need to show that the second term of the right-hand side of the above equation tends to zero as $\rho\to0$. To proceed, let us establish for example that 
\begin{equation}\label{desiredEstim1} 
\lim \limits_{\rho \rightarrow 0} \dsp \int_{\partial D_\rho} \eps\inv  \, \partial_r s^\out \,\overline{\tilde{u}} \, d \sigma = 0,
\end{equation} 
the other terms being handled in the same way. Using Green's formula, we would like to write
\begin{equation}\label{desiredEstim2} 
\dsp \int_{\partial D_\rho} \eps\inv  \, \partial_r s^\out \overline{\tilde{u}} \, d \sigma = \dsp \int_{D_\rho} \eps\inv  \, \grad s^\out \cdot \overline{\grad \tilde{u}} \, d \bfx.
\end{equation}
The difficulty here is that $s^\out \not \in H^1(D_\rho)$. But we can prove \cite{Kond67,BoChCl13} that $\tilde{u}$ has more regularity than $H^1$ regularity. More precisely, the behaviour of $\tilde{u} $ at the corner is driven by the less regular singularity associated with singular exponents $\lambda$ such that $\Re e\,\lambda>0 $. Set
\[ \beta_0 := \min \lbrace \Re e\,\lambda | \, \lambda \in \Lambda \mbox{ and } \Re e\,\lambda >0 \rbrace.\]
We can show that $r^{-\beta} \grad \tilde{u} \in L^2(D_R)$ for $\beta < \beta_0 $ (see \cite{BoChCl13}), which implies, for all $0<\beta < \beta_0 $,
\begin{equation}\label{desiredEstim3} 
 \Big| \dsp \int_{D_\rho} \eps\inv  \, \grad s^\out \cdot \overline{\grad \tilde{u}} \, d \bfx \Big| \leq C\,\Big( \dsp \int_{D_\rho} r^{2\beta} \vert \grad s^\out \vert^2 \, d\bfx \Big)^{1/2}\Big( \dsp \int_{D_\rho} r^{-2\beta} \vert \grad \tilde{u}\vert^2 \, d\bfx \Big)^{1/2}.
 \end{equation} 
Since there holds, 
\[
\dsp \int_{D_\rho} r^{2\beta} \vert \grad s^\out \vert^2 \, d\bfx  \leq \rho^\beta \dsp \int_{D_\rho} r^{\beta} \vert \grad s^\out \vert^2 \, d\bfx\  \underset{\rho\to0}{\longrightarrow}\ 0,
\]
combining (\ref{desiredEstim2})--(\ref{desiredEstim3}) leads to (\ref{desiredEstim1}).

\section*{Acknowledgments}

The research of the authors is supported by the ANR project METAMATH, grant ANR-11-MONU-016 of the French Agence Nationale de la Recherche.
The research of C. C. is supported by the DGA, Direction G\'en\'erale de l'Armement. The research of L. C. is supported by the FMJH through the grant ANR-10-CAMP-0151-02 in the ``Programme des Investissements
d'Avenir''. 

\bibliography{Bibliography}
\bibliographystyle{plain}

\end{document}